\documentclass{amsart}
\usepackage{amsmath,amssymb,amsthm}
\usepackage{enumitem}

\usepackage{amsmath}
\newcommand{\argmin}[1]{\underset{#1}{\mathrm{argmin}}}
\usepackage{algorithm}
\usepackage{algpseudocode}
\usepackage[table]{xcolor}
\usepackage{graphicx}
\usepackage{epstopdf}
\usepackage{subfigure}
\usepackage{amsmath}
\usepackage{hyperref}

\newtheorem{theorem}{Theorem}[section]
\newtheorem{proposition}[theorem]{Proposition}

\usepackage{hyperref}

\usepackage{tcolorbox}

\usepackage[font=small,labelfont=bf]{caption}

\usepackage{array}
\newcolumntype{P}[1]{>{\centering\arraybackslash}p{#1}}

\title{Learning Dynamical Systems and Bifurcation via Group Sparsity}

\author{Hayden Schaeffer}
\address{Department of Mathematical Sciences, Carnegie Mellon University, Pittsburgh, Pennsylvania, United States}
\email{schaeffer@cmu.edu}

\author{Giang Tran}
\address{Department of Applied Mathematics, University of Waterloo, Waterloo, Ontario, Canada}
\email{giang.tran@uwaterloo.ca}

\author{Rachel Ward}
\address{Department of Mathematics, The University of Texas at Austin, Austin, Texas, United States}
\email{rward@math.utexas.edu}

\begin{document}
\maketitle

\begin{abstract}

Learning governing equations from a family of data sets which share the same physical laws but differ in bifurcation parameters is challenging. This is due, in part, to the wide range of phenomena that could be represented in the data sets as well as the range of parameter values. On the other hand, it is common to assume only a small number of candidate functions contribute to the observed dynamics. Based on these observations, we propose a group-sparse penalized method for model selection and parameter estimation for such data. We also provide convergence guarantees for our proposed numerical scheme. Various numerical experiments including the 1D logistic equation, the 3D Lorenz sampled from different bifurcation regions, and a switching system
provide numerical validation for our method and suggest potential applications to applied dynamical systems.
 \end{abstract}



\section{Introduction}\label{sec:introduction}

Nonlinear systems of ordinary differential equations (ODEs) are used to describe 
countless  physical and biological processes. Often, the governing equations that model a system must be derived theoretically or computed to fit a given dataset. Model selection and parameter estimation methods are used to make decisions on the form of the governing equations and the values of the model parameters. One major difficulty is determining the set of appropriate  candidate functions to fit to the data, since the task of searching through a large set of potential candidate functions can be computationally intractable. Therefore, it is common to pre-select a small subset of potential candidate functions \cite{burnham2003model, kuepfer2007ensemble}. However, this requires prior knowledge on the system and the potential structure of the governing equations. Another issue involves the estimation of model parameters when using multiple sampling sources, since a `one-size-fit-all' approach may not be possible when the parameters in the governing equations vary over the different sources. In this work, we develop a group-sparse penalized method for model selection and parameter estimation of governing equations where the data is given over multiple sources. Possible applications include parameter analysis through data-driven bifurcation diagrams, analysis of chaotic systems, and parameter estimation from incomplete and non-uniform sources.

There have been several recent works utilizing sparse optimization for model selection. These methods were inspired by the regression approach from \cite{bongard2007automated, schmidt2009distilling}, which used a symbolic regression algorithm to learn physical laws from data by fitting derivatives of the data to candidate functions. In terms of regression approaches, sparsity can be incorporated by adding an $\ell^0$ term to the method, which penalizes the number of non-zero candidate functions in the learned model. In some cases, it is possible to relax the $\ell^0$ penalty to the $\ell^1$ norm and still maintain sparse solutions, see \cite{candes2006stable, candes2007sparsity, donoho2006most, rudelson2008sparse, rauhut2010compressive}. Both the $\ell^0$ and $\ell^1$ penalties have seen many applications from image processing to data mining, and now in methods for learning governing equations from dynamic data. The key idea of sparse model selection for learning governing equations is to find the best fit of the temporal derivative over a large set of potential candidates by enforcing that the selected model uses only a few terms (\textit{i.e.} the model should be sparse). This is based on the sparsity-of-effect principle, since one expects that a small number of candidate functions contribute to the observed dynamics. In \cite{brunton2016discovering}, a sequential least-squares thresholding algorithm was proposed for learning dynamical systems. The algorithm iterates between the least-squares solution and a thresholding step, which is meant to retain only the most meaningful terms in the least-squares approximation. 
In \cite{tran2017exact}, a joint outlier detection and model selection method was developed for learning governing equations from data with time-intervals of highly corrupted information. A group-sparse penalty is used, namely the $\ell^{2,1}$ norm, to couple the intervals of corruption between each variable while also penalizing the size of these intervals.  It was also shown in  \cite{tran2017exact} that the separation between the clean data and outliers can be exactly recovered from chaotic systems. An $\ell^1$ regularized least-squares approach (related to the  LASSO method \cite{tibshirani1996regression}) is used in \cite{schaeffer2017learning} to learn nonlinear partial differential equations from spatio-temporal data. The dictionary matrix is constructed from a set of basis functions applied to both the data and its spatial derivatives.  In \cite{rudy2017data}, a sequential least-squares method related to \cite{brunton2016discovering} was used to learn PDE from spatio-temporal data as well. In \cite{schaeffer2017sparse}, an $\ell^0$ basis pursuit problem was proposed which used an integrated candidate set to identify governing equations from noisy data. In \cite{schaeffer2017extracting}, an $\ell^1$ basis pursuit problem was developed for extracting governing equations from under-sampled data. The $\ell^1$ basis pursuit approach solves the model selection problem exactly (under certain conditions), when the data is sampled using several proposed strategies. It is important to note that the sparse optimization and data-based methods have seen many applications to various scientific problems over that last few years, see \cite{schaeffer2013sparse,mackey2014compressive,hou2015sparse, rauhut2016interpolation, adcock2017polynomial, peng2016polynomial, tran20151,caflisch2015pdes,ozolicnvs2013compressed,brunton2016discovering,bright2013compressive, berkooz1993proper,holmes2012turbulence, giannakis2015data, williams2015data, coifman2006diffusion,coifman2005geometric,nadler2006diffusion,nadler2006diffusion2, schmid2010dynamic, schmid2012decomposition,schaeffer2016accelerated,scarnati2017using}.

In several of those works, it is noted that multiple sampling sources increase the accuracy of the recovered coefficients, when the coefficients are assumed to remain fixed over all sources. This is, in part, due to the fact that over various sources, the data will exhibit different behaviors which help to distinguish between possible candidates.  In this work, we develop a group-sparse model for extracting governing equations from multiple sources, whose parameters may vary between the sources.  In particular, we enforce that the learned model is the same between the sources and allow the coefficients to vary for each source. The learned coefficients from group-sparse methods are block-wise sparse in the sense that variables are grouped into subsets which are either simultaneously zero or nonzero. This can be thought of as an `all-or-nothing' penalty within each group. Mathematically, group-sparse methods often use an $\ell^{2,0}$ or $\ell^{2,1}$ penalty to group together related variables while penalizing the number of active (nonzero) groups \cite{yuan2006model}. To solve the group-sparse optimization problem, several algorithms have been proposed such as hard-iterative thresholding algorithms \cite{gribonval2008atoms,blanchard2014greedy};  simultaneous orthogonal matching pursuits (SOMP) based on correlations  \cite{gribonval2008atoms} or noise stabilization \cite{determe2015simultaneous}, and subspace methods \cite{lee2012subspace}. Recovery guarantees for these algorithms were also investigated, for example, using a probability model on the nonzero coefficients \cite{gribonval2008atoms} or full rank conditions \cite{davies2012rank}. Some applications of group sparsity to data-based learning include microarray data analysis \cite{ma2007supervised}, spectrum cartography \cite{bazerque2011group}, and source localization \cite{ollila2015nonparametric}. Using the group-sparse penalty can lead to a decrease in the number of degrees of freedom in the inverse problem which can potentially increase the accuracy of recovery \cite{deng2012group}. This is the case for learning governing equations with multiple sources or from different bifurcation regions, as detailed in Section~\ref{section:numerics}.

The paper is outlined as follows. In Section~\ref{section:problem}, the framework for sparse model selection is detailed as well as the group-sparse structural condition and problem statement. Convergence conditions related to the numerical solver is discussed in Section~\ref{sec:coercive}. In Section~\ref{section:numerics}, the numerical method, based on the iterative hard thresholding algorithm, is explained and the computational scheme is presented. Computational results for several dynamical systems are shown in Section~\ref{section:results}. We end with some concluding remarks in 
Section~\ref{section:conclusion}.

\section{Problem Statement}\label{section:problem}

Consider the variable $x(t; \lambda^{(i)})\in \mathbb{R}^n$ governed by the nonlinear dynamical system:
\begin{equation}
\dot{x}(t;\lambda^{(i)}) = f(x(t);\lambda^{(i)}),\quad i = 1,\ldots, m.
\label{eq:ode}
\end{equation}
In vector form, $x(t; \lambda^{(i)}) = (x_1(t; \lambda^{(i)}),\ldots, x_n(t; \lambda^{(i)}))^T$ represents the state of the system at time $t$, the function $f(x;\lambda^{(i)}) = (f_1(x;\lambda^{(i)}),\ldots, f_n(x;\lambda^{(i)}))^T$ defines the nonlinear evolution of the system, and $\lambda^{(i)}$ is a bifurcation parameter of the system. We would like to learn the function $f$ and the model parameters $\lambda^{(i)}$, when only measurements on $x$ are provided. The parameters $\lambda^{(i)}$ correspond to different bifurcation regimes in the observed variables. We have no \textit{a priori} information on the functions $f$, the parameters $\lambda^{(i)}$, or the bifurcation region in which the state space $x(t;\lambda^{i})$ resides. The velocity $\dot{x}$ is assumed to be observed or calculated, with sufficient accuracy, from the states $x$.

Since the nonlinear functions $f$ are unknown, we will represent them as a linear combination of a large set of candidate nonlinear functions. This transforms the nonlinear regression problem on $f$ to a linear inverse problem with respect to the coefficients. In the construction below, we develop the proposed approach for polynomial systems; however, it can be easily generalized to other (and possibly redundant) candidate sets. Since we are given the observations of the states $x(t;\lambda^{(i)})$, Equation~\eqref{eq:ode} decouples and thus the model selection can be done component-wise. Therefore, we consider the following representation for the nonlinear function $f_j$:
\begin{equation}
f_j(x(t);\lambda^{(i)}_{j}) = c^{(i)}_{j,0} + \sum\limits_{k} c^{(i)}_{j,k}\ x_k(t;\lambda^{(i)}) + \sum\limits_{k,l} c^{(i)}_{j,k,l}\ x_k(t;\lambda^{(i)}) \, x_l(t;\lambda^{(i)}) + \ldots
\label{eq:polyrep}
\end{equation}
where  $c^{(i)}_{j} = (c_{j,0} ,c_{j,1},\ldots, c_{j,n}, c_{j,1,1},\ldots, c_{j,n,n},\ldots)$ is the vector of unknown coefficients corresponding to the $j^{th}$ component of Equation~\eqref{eq:ode} and the $i^{th}$ source. The parameters $\lambda^{(i)}_{j}$ denote the subset of parameters corresponding to $f_j$. Note that observations $x(t;\lambda^{(i)}) $ still depend on the entire system $\lambda^{(i)}$.  The model selection problem is to identify the support set of $c^{(i)}_{j}$, \textit{i.e.} the indices that correspond to the nonzero values. Since the model is assumed to have the same representation for all $i$, the support set of $c^{(i)}_{j}$ is also the same for all $i$ and thus we will denote it as $S_j$.  The parameter estimation problem corresponds to learning the nonzero values of $c^{(i)}_{j}$ for each $i$ and $j$. It is worth noting that if the correct model is identified, then the coefficients $c^{(i)}_{j}$ restricted to the set $S_j$ should be identical to $\lambda^{(i)}_{j}$. 

To illustrate and clarify the notation, consider Duffing's equation, $\ddot{u}+\delta \dot{u} -\beta u + u^3=0 $, which can be written as a nonlinear first-order system:
\begin{align*}
\dot{x}_1 &= x_2\\
\dot{x}_2 &= \beta x_1-\delta x_2 -x_1^3
\end{align*}
The model parameters are $\lambda = (1,\beta, -\delta, -1 )$, which will lead to different types of observed behaviors. The parameter subsets are $\lambda^{(i)}_{1}=1$ (which is the same for all $i$) and $\lambda^{(i)}_{2}=(\beta^{(i)}, -\delta^{(i)}, -1 )$. Consider two states $i\in \{1,2\}$. For $i=1$, if we observe data from the system with $\beta>0$ and $\delta>0$, there would be three equilibrium points: two sinks and a saddle. For $i=2$, if we observe data from the system with $\beta<0$ and $\delta>0$, there would be one sink (at the origin).  The vector $\lambda$ represents the controls on the behavior of the output and are unknown to the user. Learning the representation of $f_2$ from observations of $x$ and $\dot{x}$ over these two parameter states, if successful, would yield: 
\begin{align*}
f_2(x; \lambda^{(i)}_{2})&= 0 + c^{(i)}_{2,1}\ x_1(t;\lambda^{(i)})+ c^{(i)}_{2,2}\ x_2(t;\lambda^{(i)})+0 +\ldots + 0+ c^{(i)}_{2,1,1,1}\ x^3_1(t;\lambda^{(i)})+0 +\ldots 
\end{align*}
where the learned coefficient vector is $c^{(i)}_{2} = (0 ,c^{(i)}_{2,1},c^{(i)}_{2,2},0, \ldots, 0,c^{(i)}_{2,1,1,1} ,0, \ldots)$. The support set of $c^{(i)}_{2}$ is the same over all $i$ and, if the representation is accurate, the restriction of $c^{(i)}_{2}$ onto the nonzero values,  $(c^{(i)}_{2,1},c^{(i)}_{2,2}, c^{(i)}_{2,1,1,1})$, should match $\lambda^{(i)}_{2}=(\beta^{(i)}, -\delta^{(i)}, -1 )$ for the two regimes $i\in \{1,2\}$.

The goal is to recover $c^{(i)}_{j}$ by fitting each $f_j$ to the observed or calculated velocity. The inverse problem can be written as a linear system with unknowns $c^{(i)}_{j}$ as follows. For a fixed parameter state $\lambda^{(i)}$, we denote the observed variables at times $\{t^{(i)}_1, ... , t^{(i)}_{\ell_i}\}$ ($\ell_i$ is the number of temporal measurements obtained from state $i$) as 
\begin{equation}
\{x(t^{(i)}_1;\lambda^{(i)}), \ x(t^{(i)}_2;\lambda^{(i)}),\ldots,\ x(t^{(i)}_{\ell_i};\lambda^{(i)})\}.
\label{eqn:points}
\end{equation}
Similar to \cite{brunton2016discovering,tran2017exact,schaeffer2017sparse,schaeffer2017learning,rudy2017data, schaeffer2017extracting}, the data matrix $X^{(i)}$, the velocity matrix $V^{(i)}$, and the dictionary matrix $D^{(i)}$ are defined as:

\begin{align*}
   X^{(i)} &= \begin{bmatrix}
       |  &   |     &           & |      \\
   x^{(i)}_1 & x^{(i)}_2 & \ldots & x^{(i)}_n  \\
       |  &  |      &            & |      \\
   \end{bmatrix}_{\ell_i\times n} 
   = \begin{bmatrix}
          {x}_1(t_1;\lambda^{(i)}) & {x}_2(t_1;\lambda^{(i)}) &\ldots & {x}_n(t_1;\lambda^{(i)})\\
           {x}_1(t_2;\lambda^{(i)}) & {x}_2(t_2;\lambda^{(i)})&\ldots & {x}_n(t_2;\lambda^{(i)})\\
           \vdots &\vdots&\ldots & \vdots\\
          {x}_1(t_{\ell_i};\lambda^{(i)}) &{x}_2(t_{\ell_i};\lambda^{(i)}) &\ldots & {x}_n(t_{\ell_i};\lambda^{(i)})\\
         \end{bmatrix}_{\ell_i\times n},
\end{align*}
\begin{align*}
   V ^{(i)} &= \begin{bmatrix}
       |  &   |     &           & |      \\
   \dot{x}^{(i)}_1 & \dot{x}^{(i)}_2 & \ldots & \dot{x}^{(i)}_n  \\
       |  &  |      &            & |      \\
   \end{bmatrix}_{\ell_i\times n} 
   = \begin{bmatrix}
           \dot{x}_1(t_1;\lambda^{(i)}) & \dot{x}_2(t_1;\lambda^{(i)}) &\ldots & \dot{x}_n(t_1;\lambda^{(i)})\\
           \dot{x}_1(t_2;\lambda^{(i)}) & \dot{x}_2(t_2;\lambda^{(i)})&\ldots & \dot{x}_n(t_2;\lambda^{(i)})\\
           \vdots &\vdots&\ldots & \vdots\\
           \dot{x}_1(t_{\ell_i};\lambda^{(i)}) & \dot{x}_2(t_{\ell_i};\lambda^{(i)}) &\ldots & \dot{x}_n(t_{\ell_i};\lambda^{(i)})\\
         \end{bmatrix}_{\ell_i\times n},
\end{align*}

\noindent and
\begin{align*}
       D^{(i)}  &= \left[\textbf{1}_{\ell_i,1}, \ X^{(i)} , \ (X^{(i)})^2 , \ (X^{(i)})^3,\ \ldots \right]_{\ell_i\times \overline{n}},
  \end{align*}
 where $\overline{n} = {n+p \choose n}$ is the total number of monomials up to degree $p$. Each column of these matrices corresponds to the vectorization of each variable over the temporal measurements. For each $q$, $(X^{(i)})^q$ denotes the values of all monomials of degree $q$ at times $\{t^{(i)}_1, ... , t^{(i)}_{\ell_i}\}$. 
 
For every index $j\in \{1,\ldots,n\}$ (the components of the system), the problem of finding $f_j(x;\lambda^{(i)}_j)$, for all $ i = 1,\ldots, m$, can be reformulated to finding $c_j^{(i)}$ given the data matrix $X^{(i)}$. In particular, the problem can be stated as: find $c_j^{(i)}\in \mathbb{R}^{\overline{n}}$ such that:
\begin{equation}
V_j^{(i)} = D^{(i)} c_j^{(i)},\quad i=1,\ldots, m,
\label{eq:invpro}
\end{equation}
where $V_j^{(i)}$ is the $j^{th}$ column of $V^{(i)}$. 

Next, define the coefficient matrix by:
\begin{align*}
C_j = \begin{bmatrix}
|  &  | & |      & |  &    \\
c_j^{(1)} & c_j^{(2)} & \ldots & c_j^{(m)} \\
|  &  | & |      & |  &    \\
\end{bmatrix}_{\overline{n}\times m}
\end{align*}
where each column corresponds to the vector $c_j^{(i)}$ for a fixed $i$ and each row corresponds to the same candidate term in the representation of the $f_j$'s. Let $D$ be a block diagonal matrix whose diagonal corresponds to the matrices $D^{(i)}$, for $i = 1,\ldots, m$. Let $C^{vec}_j:=[c^{(1)}_j,\cdots, c^{(m)}_j]$ be the column-based vectorization of the coefficient matrix $C_j$, and let $V^{vec}_j: =[V^{(1)}_j,\cdots, V^{(m)}_j] $ be the column-based vectorization of the velocities of component $j^{th}$ along each source. Then the optimization problem can be rewritten as a least-square fitting :
 \begin{equation}
 { \min\limits_{C_j} \| DC^{vec}_j -V^{vec}_j \|_2^2  \left(= \sum\limits_{i=1}^m\| D^{(i)} c_j^{(i)} -V_j^{(i)} \|_2^2\right)}
 \label{eqn:invpro2}
\end{equation}

This is ill-posed due to errors in the measurements $x$, errors in approximating the velocity, and issues involving the large set of candidate functions (which would lead to overfitting if one solves Equation~\eqref{eqn:invpro2} using the pseudo-inverse, for example). 

To regularize the problem, we include a penalty on the number of active candidate functions. This will help to prevent overfitting and lead to meaningful result in practice. The assumption on the system is that $c_j^{(i)}$ has the same support set (in $j$) for each $i$, but can differ in value. In other words, we can group each row together to be either zero or nonzero, therefore the number of active (nonzero) rows is sparse. This leads to the following group-sparse optimization problem:
 \begin{equation}
 \boxed { \min\limits_{C_j}\  \| DC^{vec}_j -V^{vec}_j \|_2^2 + \gamma \|C_j\|_{2,0}  \quad \Leftrightarrow  \quad \min\limits_{C_j}\  \sum\limits_{i=1}^m\| D^{(i)} c_j^{(i)} -V_j^{(i)} \|_2^2 +  \gamma \|C_j\|_{2,0}}
  \label{model:nonconvex}
\end{equation}
where the $\ell^{2,0}$ penalty is defined as:
\begin{equation*}
\|A\|_{2,0}: = \# \left\{ k: \left( \sum_{\ell} |a_{k,\, \ell}|^2 \right)^{1/2} \neq 0\right\}.
\end{equation*}
for any matrix $A=[\, a_{k,\, \ell}\, ]$. Although the problem is nonconvex, we can solve it numerically using an iterative hard thresholding algorithm, see Section~\ref{section:numerics}.

\section{Convergence Guarantees}\label{sec:coercive}

The addition of the $\ell^{2,0}$ penalty is to encourage sparse solution from the least-squares minimization. If the matrix $D$ is not full rank or badly conditioned, then Equation~\eqref{model:nonconvex} is not guaranteed to have a unique solution. Indeed, our dictionary matrix, formed by monomials up to possibly high order, will generally not be well-conditioned.  In Proposition \ref{prop:coercive1} Section~\ref{sec:coercive}, we characterizes properties of the dynamics such that $D$ (or equivalently, each $D^{(i)}$) will at least be full rank. In fact, our proposed numerical method is guaranteed to converge to a local minimizer of Equation~\eqref{model:nonconvex} if $D$ is full rank (see Section~\ref{section:numerics} and Appendix for more details).

\begin{proposition}[General bound]
Suppose, for each $i$, $\overline{n} \leq \ell_i$, there exists a subset $S \subset [\ell_i]$ of size $| S | = \overline{n}$ such that $\{X^{(i)}(k,-)\mid k\in S\}$ do not belong to a common algebraic hypersurface of degree $\leq p$ (i.e., for any $u_1, u_2, \dots, u_{\overline{n}} \in \mathbb{R}^n$, there exists a unique interpolating polynomial $u = P(x)$ of degree $\leq p$ satisfying $u_k = P(X^{(i)}(k,-))$ for $k\in S$). 
This is a necessary and sufficient condition for the dictionary matrix $D$ to be full rank: for each $D^{(i)}$, there exists a $\delta_i > 0$ such that 
\begin{equation}
\inf_{u} \frac{\| D^{(i)} u \|_2 }{ \| u \|_2} \geq \delta_i.
\end{equation}   
\label{prop:coercive1}
\end{proposition}

In particular, for 1D systems, the conditions above states that the dictionary is full rank if there is a subset of $\overline{n}$ distinct points $x^{(i)}$.

\begin{proof}
First consider the one-dimensional case $n=1$. In this case,  $D^{(i)}_{S}$, the $\overline{n} \times \overline{n}$ dictionary matrix restricted to the rows indexed by $S$, is a square
Vandermonde matrix.  Moreover, in this case, the condition that $\overline{n} = p+1$ points in $S$ admit a unique interpolating polynomial of degree $\leq p$ is equivalent to the condition that the $p+1$ points in $S$ are distinct.  It is well-known that a Vandermonde matrix such as $D^{(i)}_{S}$ is invertible if and only if its generating points are distinct; thus, $\inf_{u} \frac{\| D^{(i)}_S u \|_2 }{ \| u \|_2} \geq \delta,$ and hence also  $\inf_{u} \frac{\| D^{(i)} u \|_2 }{ \| u \|_2} \geq \delta.$

For the general $n$-dimensional case, the $\overline{n} \times \overline{n}$ matrix $D^{(i)}_{S}$ is a \emph{generalized Vandermonde matrix}, and the stated conditions are necessary and sufficient for $D^{(i)}_S$ to be nonsingular, according to Theorem 4.1  in \cite{olver2006multivariate}. Thus, again, $\inf_{u} \frac{\| D^{(i)}_S u \|_2 }{ \| u \|_2} \geq \delta,$ and hence also  $\inf_{u} \frac{\| D^{(i)} u \|_2 }{ \| u \|_2} \geq \delta.$
\end{proof}

\begin{figure}[h!]
\centering
\includegraphics[width = 3.in]{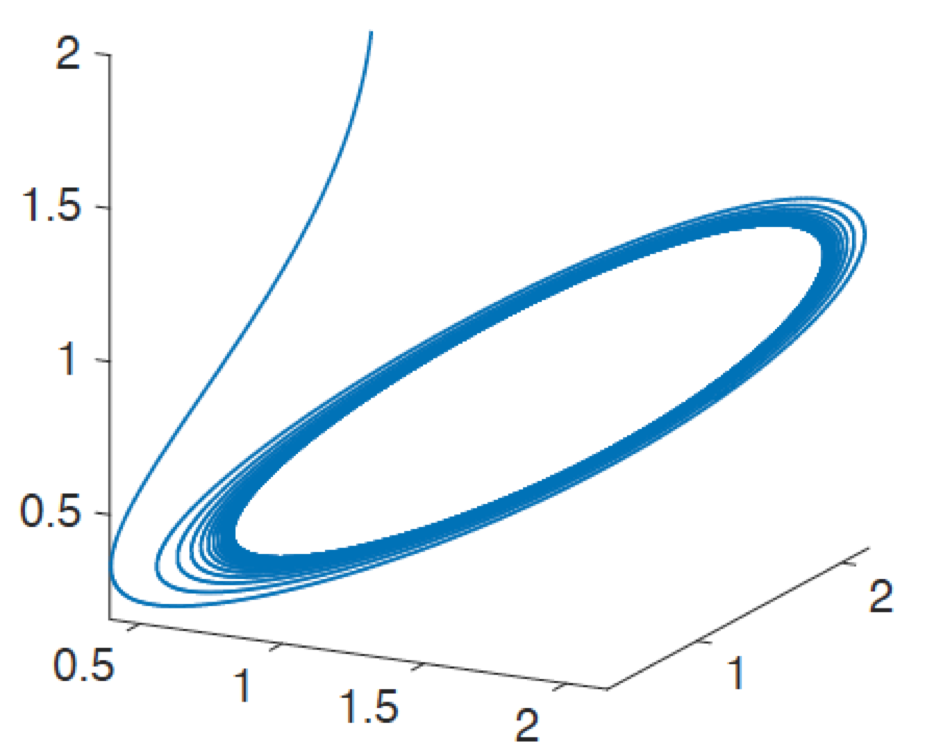}\
\includegraphics[width = 3.in]{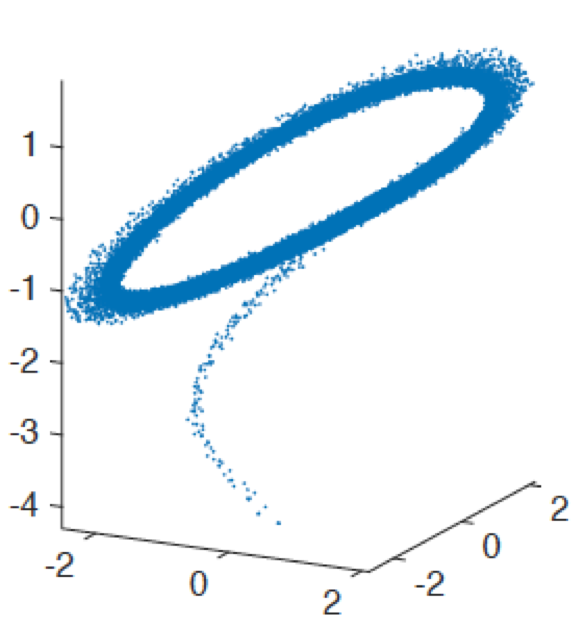}
\caption{An example where the state space quickly approaches a limit cycle, which almost stays on a hypersurface of degree 2. The state space is generated from the Lorenz system, Equation~\eqref{eqn:lorenz} with $\mu = 7.73$ and initialization $U_0= [1,1,2]$, time step $dt = 0.005$.}
\label{fig:LimitCycleFail}
\end{figure}

In Theorem~\ref{thrm:convergence}, we show that if the matrix $D$ is full rank, our proposed numerical method is guaranteed to converge to a local minimizer of the objective function. Still, the local minimizer we converge to could be far from the global minimizer.  An example of when this approach may fail is the case of a limit cycle example (see Figure \ref{fig:LimitCycleFail}). The problem is that if we initialize near the limit cycle, the dynamics lie very close to an algebraic hypersurface of degree $p=2$. This can be avoided if one could sample data away from the limit cycle.

Also, our algorithm likely converges to a local minimizer under weaker conditions than those in the above proposition; namely, the convergence requires that $D$ is coercive on sparse subsets. So $D$ can in fact be an underdetermined matrix, where we observe a number of snapshots smaller than the size of the dictionary (see Appendix). It remains an open problem to characterize general conditions on the dynamics under which this weaker condition holds.  At least for 1D dynamical systems, we can provide such a characterization.

\begin{proposition}
Consider the case $n=1$, thus $\overline{n} = p+1$.  Suppose that $\ell_i \geq s$, and at least $s$ of the $\ell_i$ points in set Equation~\eqref{eqn:points} are positive and distinct.  Then, restricted to column subsets $S \subset \overline{n}$ of size $|S| \leq s$, the dictionary matrix $D^{(i)}$ is coercive, i.e. there exists a $\delta > 0$ such that:
$$
\min_{S: |S| \leq s} \ \inf_{u \in \mathbb{R}^s}\ \frac{\| D^{(i)}_S u \|_2 }{ \| u \|_2} \geq \delta.
$$   
\label{prop:coercive2}
\end{proposition}

\begin{proof}
Without loss of generality, we assume that $\ell_i =  s$ and that all the points $\{ x(t^{(i)}_k;\lambda^{(i)}), k \in [s]\},$ are distinct and positive.  
Consider any subset $S \subset [\overline{n}]$ of size $|S| \leq s$; the corresponding $s \times s$ submatrix $D^{(i)}_S$ is of the form
$$ 
\left[
\begin{array}{cccc}

z_1^{a_1} & z_1^{a_2} & \dots & z_1^{a_s} \\
z_2^{a_1} & z_2^{a_2} & \dots & z_2^{a_s} \\
 &  & \dots &  \\
 z_s^{a_1} & z_s^{a_2} & \dots & z_s^{a_s}
\end{array}
\right]
$$
where $0 \leq a_1 < a_2 < \dots < a_s$ are integers and $0 < z_1 < z_2 < \dots < z_s$. This type of matrix is a so-called \emph{generalized Vandermonde matrix} and is known to be a totally positive matrix, hence invertible (see, e.g. \cite{demmel2005accurate}). Thus, 
 $\inf_{u \in \mathbb{R}^s} \frac{\| D^{(i)}_S u \|_2 }{ \| u \|_2} \geq \delta_S > 0.$
 Since there are finitely many such subsets, we take $\delta$ to be the minimum $\delta_S$ over all subsets to obtain 
  $$\min_{S: |S| \leq s} \inf_{u \in \mathbb{R}^s} \frac{\| D^{(i)}_S u \|_2 }{ \| u \|_2} \geq \delta_S > 0.$$   
\end{proof}

\section{Numerical Method}\label{section:numerics}

In this section, we detail the numerical method used in this work to solve Equation~\eqref{model:nonconvex}. In particular, we use an iterative thresholding algorithm, which keeps all the coefficients above a certain threshold (determined by the $\ell^2$ norm of each row). Each iteration  contains one gradient descent step and one reduced least-squares problem. This is similar to \cite{foucart2011recovering}, which keeps the $s$-largest rows over each iteration. To the best of our knowledge, the following approach does not appear in the literature, thus for the sake of completeness, we detail the numerical method here and include the corresponding proofs in the Appendix.

For the discuss below, we rescale $D$ so that maximum spectral norm (over $i$) of the matrices $(D^{(i)})^TD^{(i)}$ is less than or equal to 1. For simplicity, we drop the subscript $j$ in Equation~\eqref{model:nonconvex}. Denote the product between $D$ and $A$ by:
\[D\star A:=\left[D^{(1)} A_{- \,,1},\, D^{(2)} A_{-,2},\, \cdots, D^{(m)} A_{-,m}\right]\in \mathbb{R}^{\ell_1} \times \mathbb{R}^{\ell_2}\times\cdots\times \mathbb{R}^{\ell_m},\quad \text{for}\quad A\in\mathbb{R}^{\overline{n}\times m}\]
and define $V :=\left[V^{(1)},\, V^{(2)},\,\cdots, V^{(m)}\right]$. Then we can replace the least-squares term in Equation~\eqref{model:nonconvex} by:
 \[\|D\star C-V\|_2^2: = \sum\limits_{i=1}^m\| D^{(i)} c^{(i)} -V^{(i)} \|_2^2.\]
Let $F$ be the objective function in Equation~\eqref{model:nonconvex}:
 \begin{equation}
F(C) := \| D\star C -V \|_2^2 +  \gamma \|C\|_{2,0}.
\end{equation}
and let $F^*$ be a surrogate function for $F$:
 \begin{equation}
F^*(C,B) := \| D\star C-V\|_2^2 -\| D\star (C -B) \|_2^2  +\|C -B\|_2^2+  \gamma \|C\|_{2,0}.
\label{eq:sur}
\end{equation}
These functions agree when $B=C$, i.e. $F^*(C,C)=F(C)$. The numerical scheme is based on minimizing the surrogate function $F^*$. Equation~\eqref{eq:sur} can be simplified to:
 \begin{align*}
F^*(C,B) &= \|C\|_2^2 - 2\left<C, B + D^T\star (V- D\star B)\right> + \gamma \|C\|_{2,0} + \|B\|_2^2 + \|V\|_2^2 - \|D\star B\|_2^2, 
\end{align*}
where $\left<\cdot,\cdot\right>$ is the summation of the component-wise multiplication of each block.
To minimize the surrogate function with respect to $C$, we only need to minimize the first three terms. The $\ell^{2,0}$ penalty is row-separable, thus we can consider the two case for each row of $C$: either the row is zero or nonzero. Denote the support set by:
\[S: =\{k: \|C_{k,-}\|_2\neq0\}\]
and let $C_{S}$ be the coefficients restricted onto the support set. Then on rows $S$, the penalty is constant and the minimizer satisfies: 
 \begin{equation}
C_S = \left(B + D^T\star (V - D\star B)\right)_S.
\label{eq:form1}
\end{equation}
To decrease the surrogate function, one chooses between Equation~\eqref{eq:form1} or setting the row to zero (this is verified in the Appendix). This process yields:
 \begin{equation}
C = H_{\sqrt{\gamma}}\left(B + D^T\star (V - D\star B)\right),
\label{eq:minimizer}
\end{equation}
where the thresholding function is defined as:
 \begin{equation}
H_{a}(x)=\begin{cases}
&0,\quad\text{if} \quad \|x\|_2 \leq a\\
& x,\quad\text{otherwise}.
\end{cases}
\label{eq:thresholding}
\end{equation}
and applied row-wise to a matrix. We define an iterative thresholding algorithm using Equations~\eqref{eq:minimizer} and \eqref{eq:thresholding} by:
 \begin{equation}
C^{k+1} = H_{\sqrt{\gamma}}\left(C^{k} +D^T \star \left(V  - D\star C^{k}\right)\right).
\label{eq:updatematrix1}
\end{equation}
 Like many proximal descent methods, Equation~\eqref{eq:updatematrix1} may converge slowly in practice. To adjust the convergence rate, we include an additional step:
 \begin{align}
\boxed{ \begin{cases}
S^{k+1} &= \text{supp}\left(H_{\sqrt{\gamma}}\left(C^{k} +D^T \star \left(V  - D\star C^{k}\right)\right) \right) \\
C^{k+1}&= \argmin{C}\, \| D\star C-V \|_2^2 \quad \text{s.t.} \ \ \text{supp}(C)\subset S^{k+1}
\end{cases}}
\label{eq:updatematrix}
\end{align}
Note that the second step is column-wise separable and the row-support set of each column of $C^{k+1}$ is a subset of $S^{k+1}$, i.e. $\text{supp}(c^{(i)})\subset S^{k+1}$. Therefore, we can solve each reduced least-squares problem in parallel. Indeed, for each column $i$, we solve
\begin{equation}
c^{k+1,i} = \argmin{c}\| D^{(i)}\, c-V^{(i)} \|_2^2 \quad \text{s.t.} \ \ \text{supp}(c^{(i)})\subset S^{k+1}.
\end{equation}
A summary of the proposed algorithm is described below. 

\medskip

\noindent\fbox{%
\begin{minipage}{\dimexpr\linewidth-2\fboxsep-2\fboxrule\relax}
\begin{algorithmic}
\State \underline{\textbf{Group Hard-Iterative Thresholding Algorithm for Dynamical Systems}}

\medskip

\State Given: initialization matrix ${C}^0, tol $ and parameters $\gamma$.
\medskip
\While{$\|C^{k+1}-C^{k}\|_{\infty}> tol$}
\medskip
\State \textbf{for} $i = 1$ \textbf{to} $m$:
\medskip
\State \hspace{0.5 cm} $\left(\widetilde{c^{(i)}}\right)^ {k+1}= \left(c^{(i)}\right)^ {k} -  (D^{(i)})^T\left(D^{(i)}\left(c^{(i)}\right)^{k} - V^{(i)}\right)   $
\medskip
\State \textbf{end for}
\medskip

\State $S^{k+1} = \text{supp}\left(H_{\sqrt{\gamma}} \left[\widetilde{c^{(1)}}, \widetilde{c^{(2)}},\cdots,\widetilde{c^{(m)}}\right]\right) $

\medskip

\State  \textbf{for} $i = 1$ \textbf{to} $m$:
\medskip
\State\hspace{0.5 cm} $(c^{(i)})^{k+1}= \argmin{c^{(i)}}\| D^{(i)} c^{(i)} -V^{(i)} \|_2^2 \quad \text{s.t.} \ \ \text{supp}(c^{(i)})\subset S^{k+1}$.
\medskip
\State \textbf{end for}
\medskip
 \EndWhile
\end{algorithmic}
\end{minipage}%
}
 
\medskip

To give an indication of the behavior of the modified scheme for $\ell^{2,0}$ regularized least-squares minimization, we have the following theorem. 

\begin{theorem}
Let $F= \| D\star C -V \|_2^2 +  \gamma \|C\|_{2,0}$ and let $C^n$ be the sequence generated by Equation~\eqref{eq:updatematrix}, then $F(C^{n+1}) \leq F(C^{n})$ and there are subsequences that converge to local minimizers. In addition, if $D$ is coercive then the sequence $C^{n}$ converges to a local minimizer.
\label{thrm:convergence}
\end{theorem}

The proof is in the appendix and follows a similar approach to \cite{blumensath2008iterative}. When the sparsity level $s$ can be determined or estimated \textit{a priori}, one could use a modified method based on the group thresholding method of \cite{foucart2011recovering}. Specifically, at every iteration, we keep $ks$ indices, where $k>1$, corresponding to $ks$-largest rows with respect to the $\ell^2$ norm and solve the linear regression on that subset. At the final step, we keep exactly $s$ indices instead of $ks$ and follow the same process. Although similar in nature, the sequential thresholding algorithm found in \cite{brunton2016discovering} differs from the proposed  algorithm. In particular, the thresholding here is performed on a gradient descent step rather than the psuedo-inverse.

\section{Computational Results}\label{section:results}

\textbf{Logistic Equation.} For the first computational text, the proposed model (Equation~\eqref{model:nonconvex}) is applied to data generated from the 1D logistic equation
\begin{equation}
\dot{x} = f(x) :=\alpha x(1-x), \quad t\in[0,50.0],\quad x(0) = 0.01,\label{eqn:logistic}
\end{equation}
 where $\alpha$ is the bifurcation parameter. The simulated data are obtained from two sets, which are generated from Equation~\eqref{eqn:logistic} with $\alpha = 0.05$ and $\alpha=0.23$, respectively.
For both data sets, we set the time-step to $dt = 0.005$. The simulated data and the noisy velocities are plotted in Figure \ref{fig:logistic}. In all of our examples, the velocity data $V^{(i)}$ are approximated from $X^{(i)}$ using the central difference.

 \begin{figure}
 \centering
 \includegraphics[width = 2.5 in]{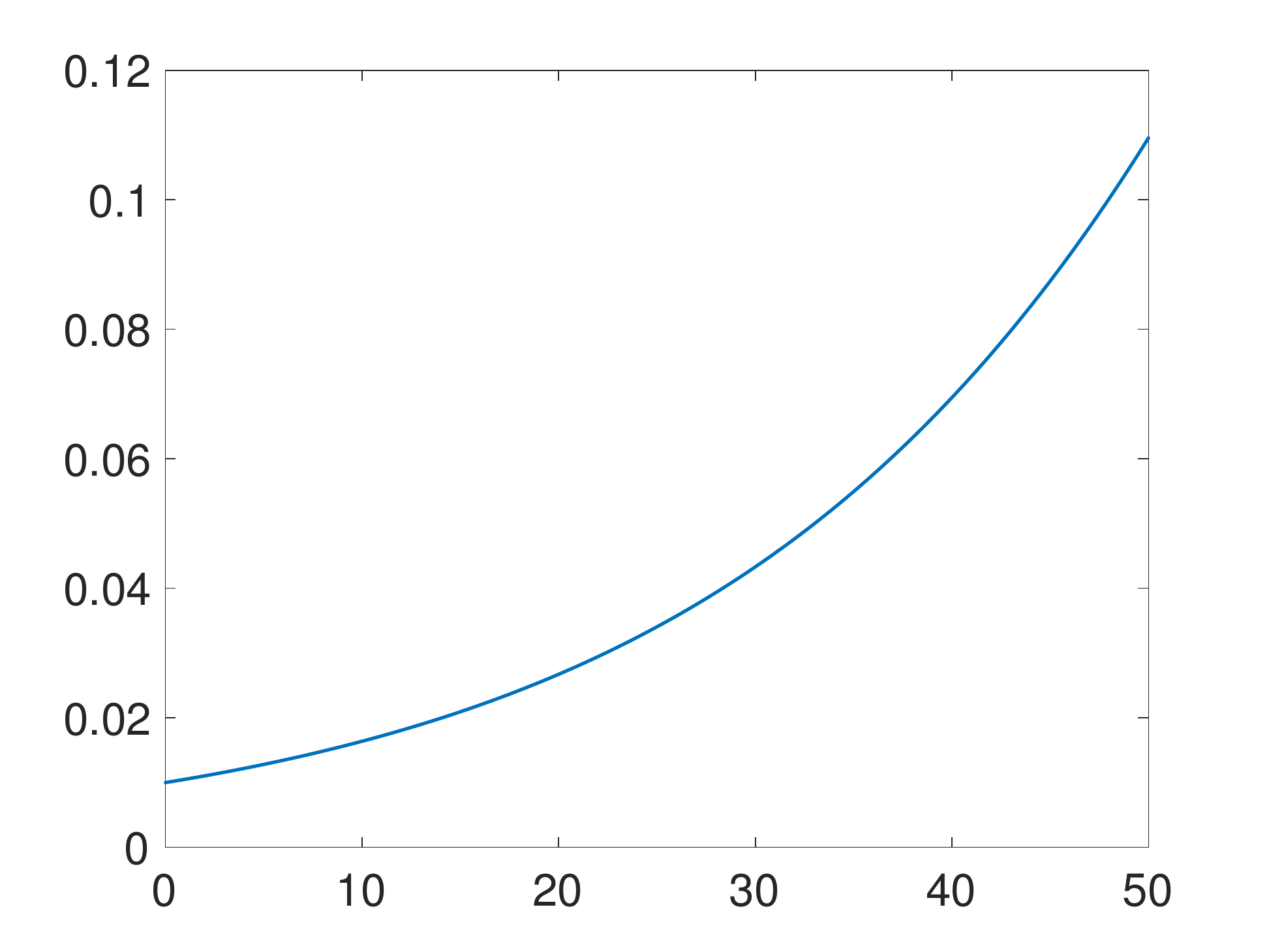}
 \includegraphics[width = 2.5 in]{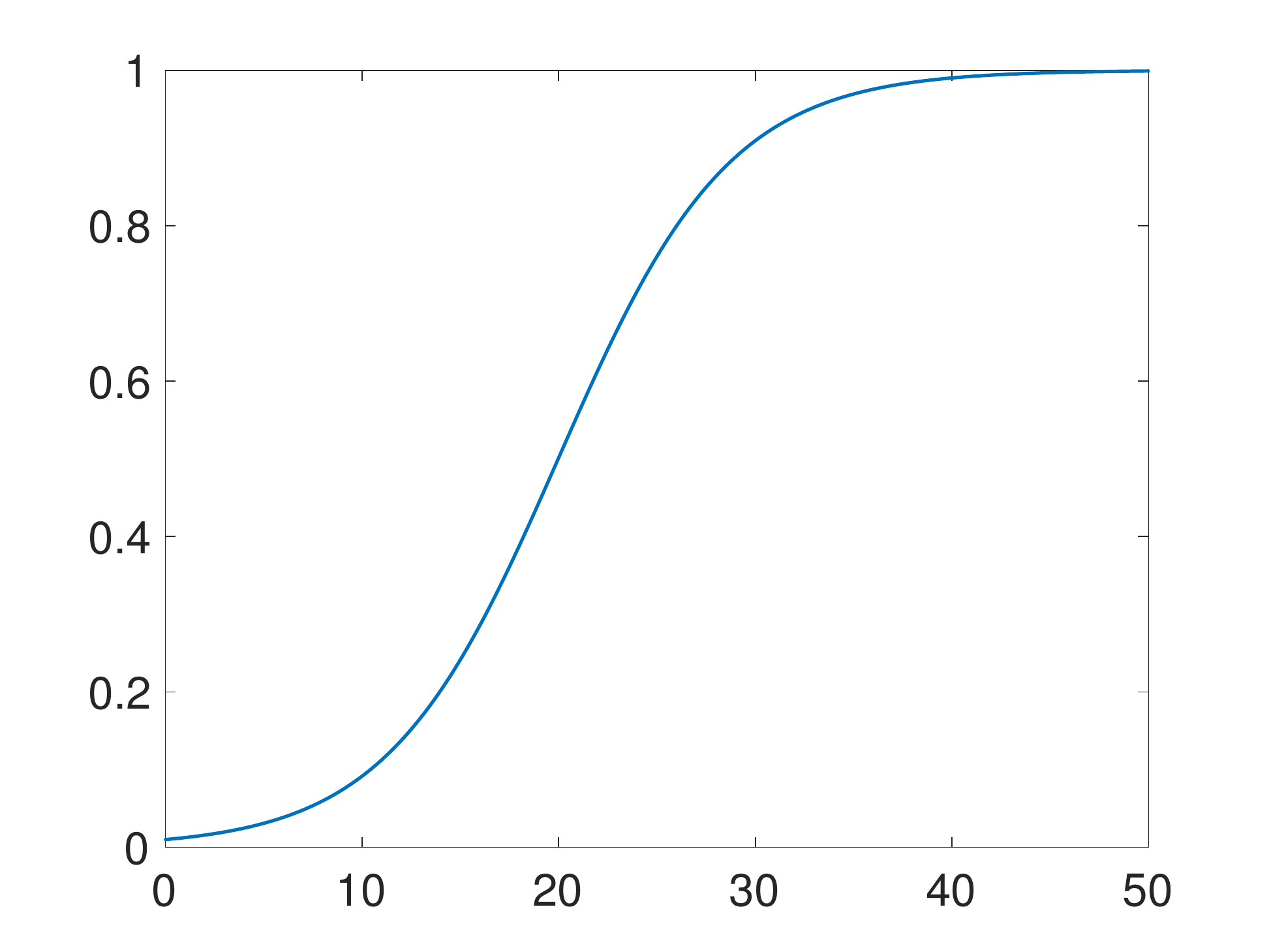}

 \includegraphics[width = 2.5 in]{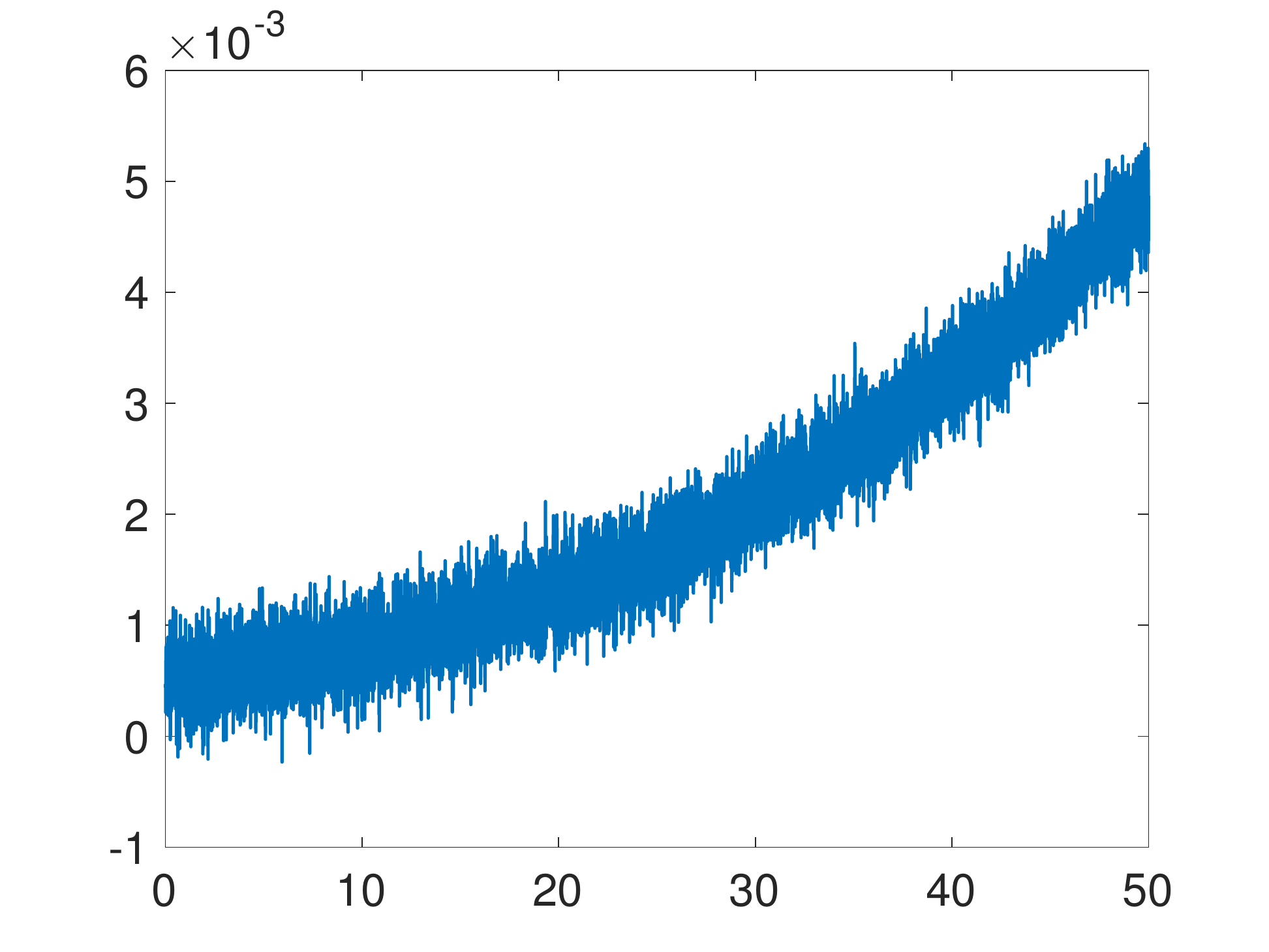}
 \includegraphics[width = 2.5 in]{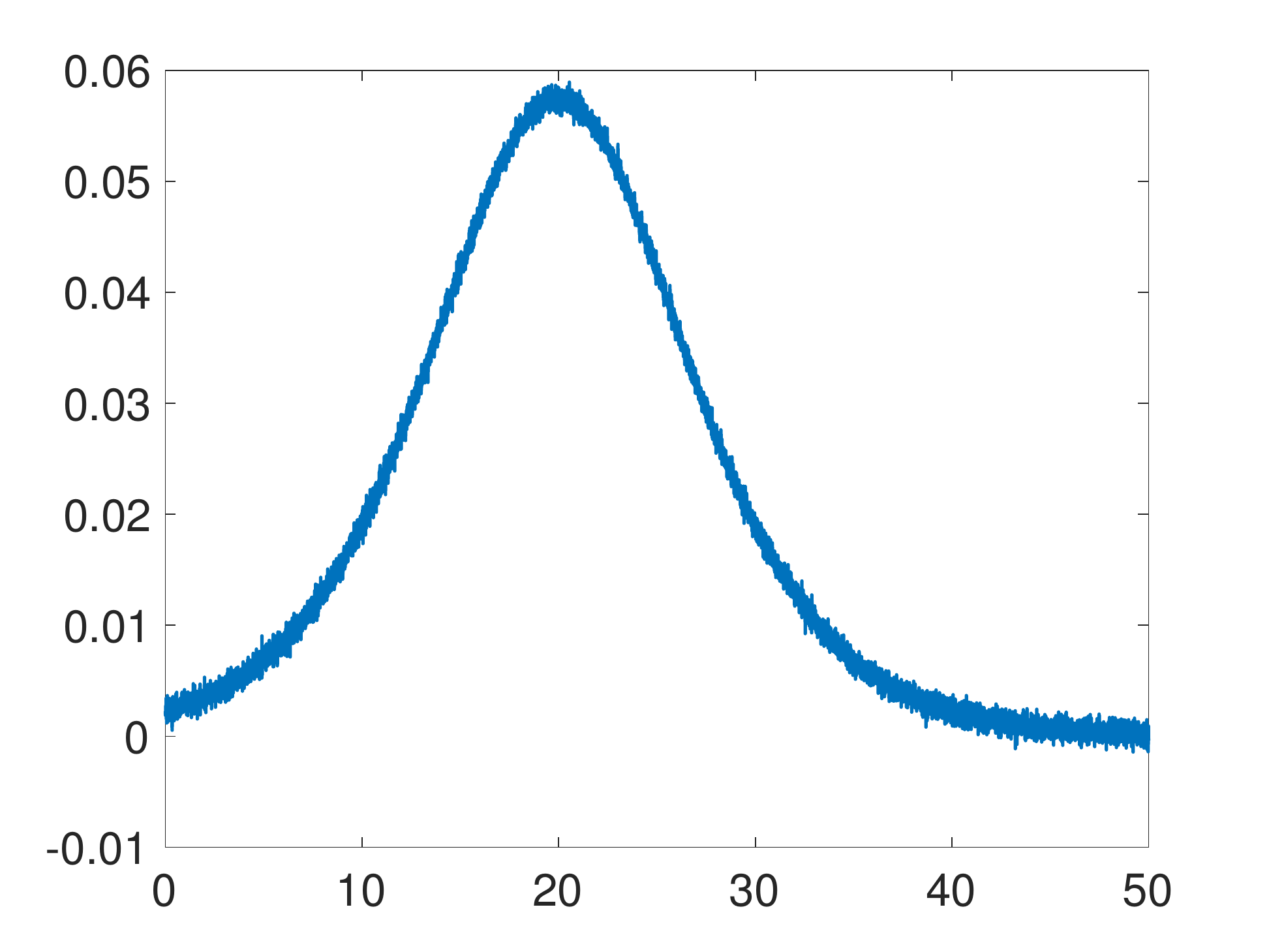}
  \caption{Logistic equation: State space (top) and noisy velocity space (bottom) plots for $\alpha= 0.05$ (left) and $\alpha = 0.23$ (right). The noise levels are $\sigma_{noise} =0.05\%$ and $\sigma_{noise} = 0.01\%$, respectively. The maximal degree of monomials in the dictionary is six. }
 \label{fig:logistic}
 \end{figure}

To test the robustness of our proposed model, we re-simulate the data 100 times and compute the probability ($P\in[0,1]$) of recovering the correct terms in the governing equation. We use a fixed thresholding parameter of $\delta_{thres} = 0.0018$. With the group-sparsity penalization, our model recovers the correct governing equation with probability $P=1$, i.e.  the method learns both terms: $x$ and $x^2$. Moreover, the average relative errors between the true coefficients and our approximations are $3.04\%$ for set 1 and around $0.02\%$ for set 2.

For comparison, if one used the $\ell_0$-penalty in place of the $\ell^{2,0}$-penalty in Equation~\eqref{model:nonconvex}, then the computed probability that the $\ell_0$-penalized method recovers the governing equation reduces dramatically to $P=0.41$ and $P=0.3$, respectively. Thus, unlike our method, the $\ell_0$ model will likely misidentify the terms in the governing equation.
\bigskip

\begin{figure}[h!]
\centering
\includegraphics[width  = 2.1 in]{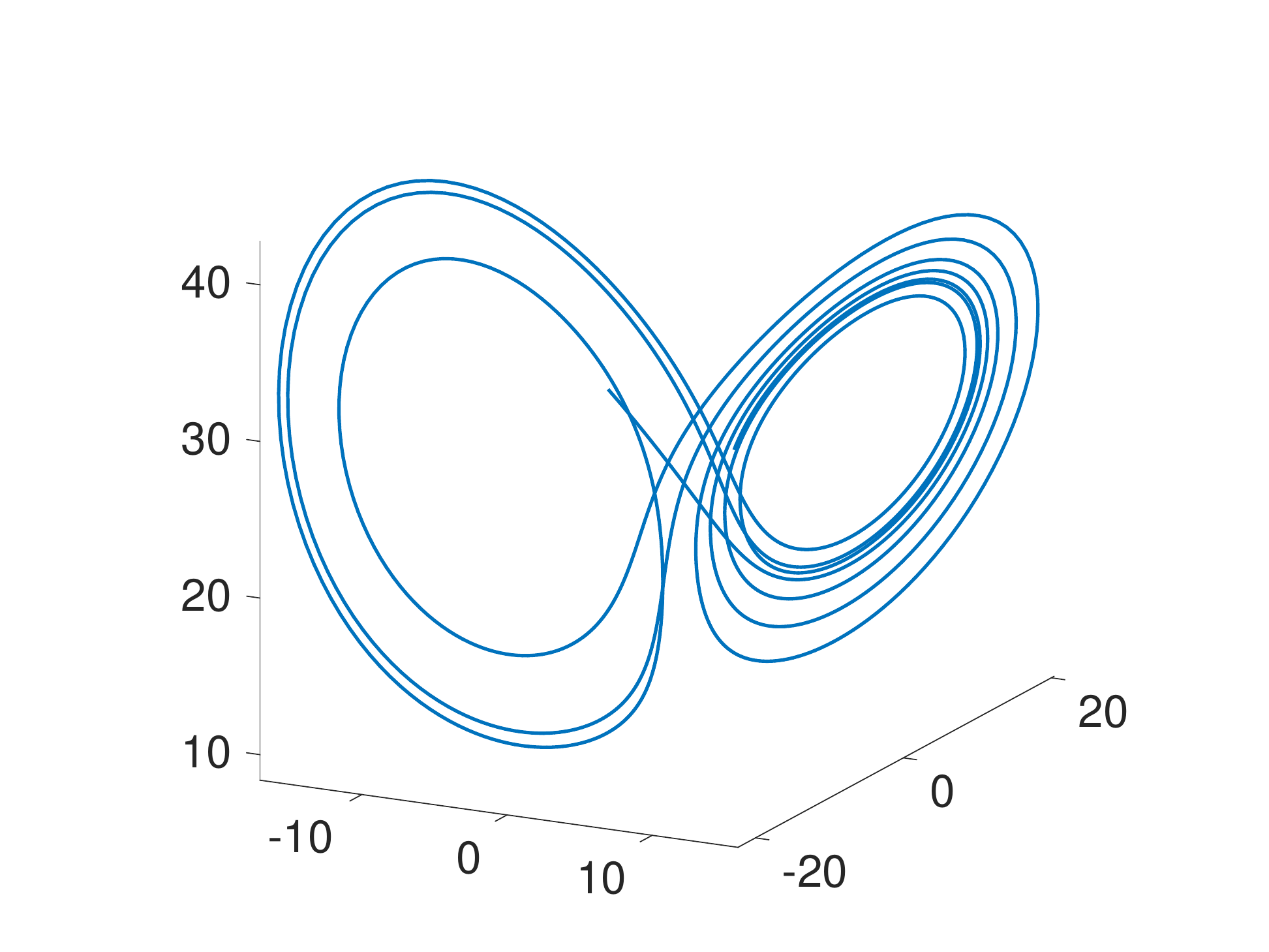}
\includegraphics[width  = 2.1 in]{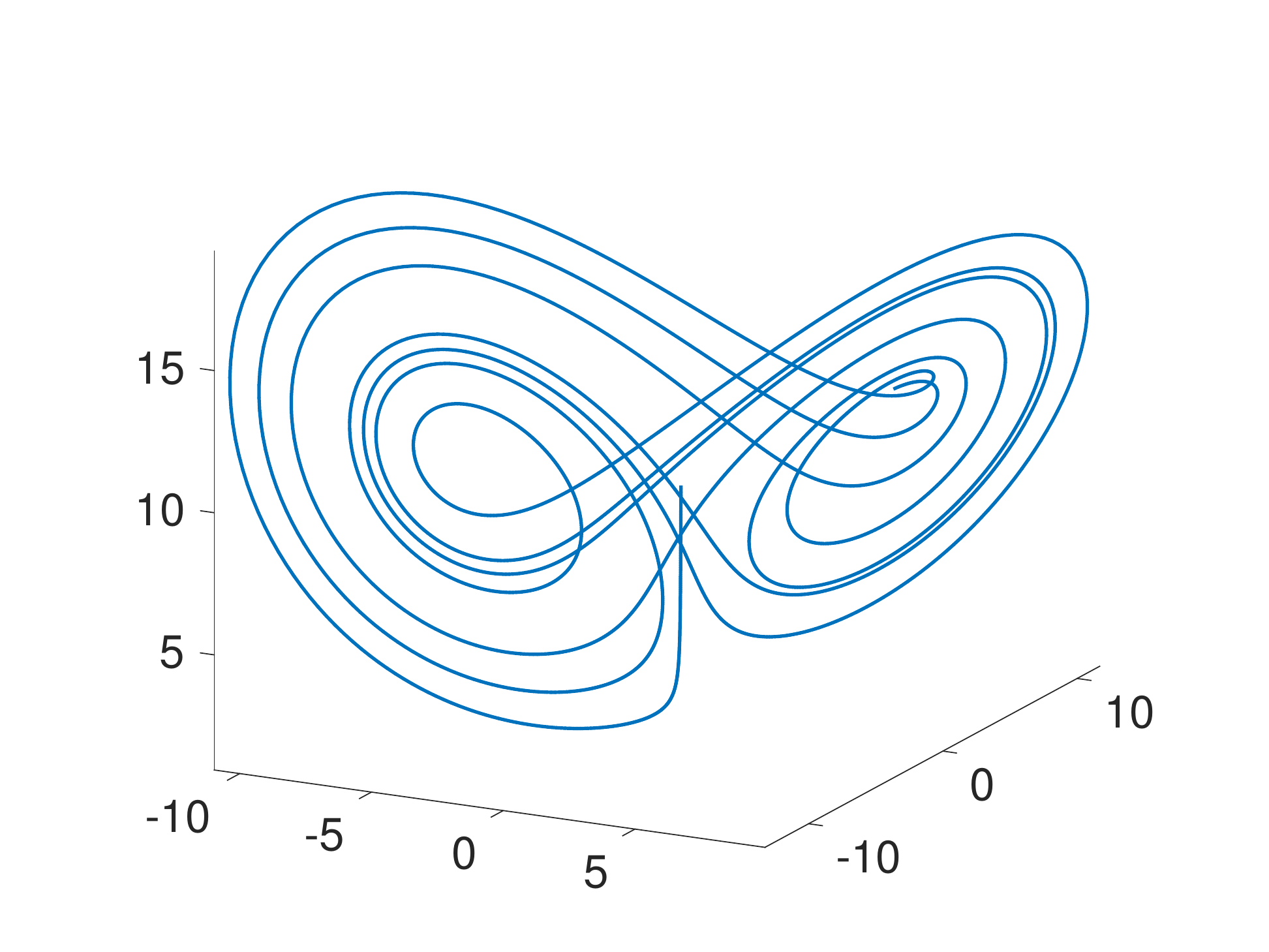}
\includegraphics[width  = 2.1 in]{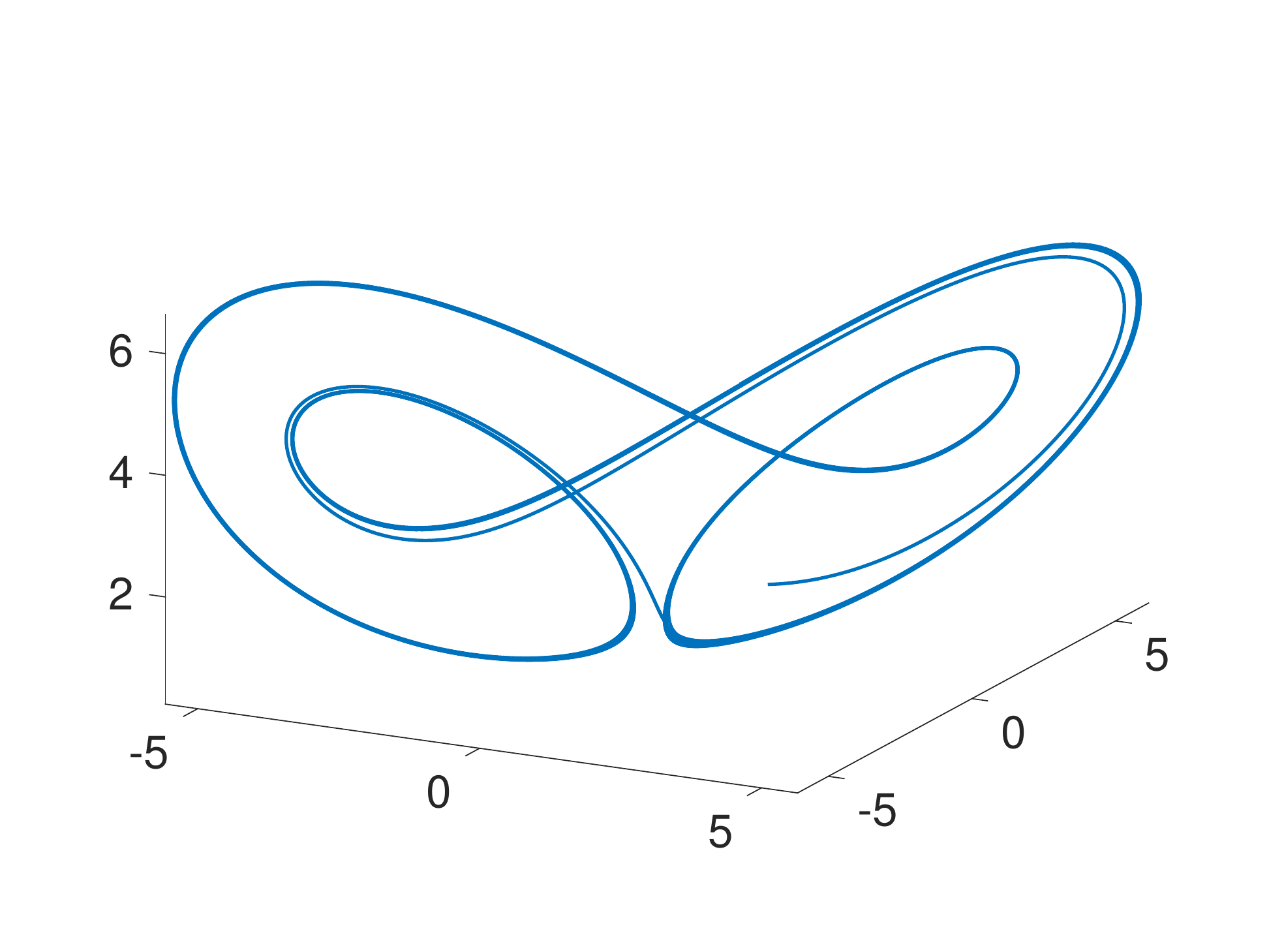}\
\includegraphics[width  = 2.1 in]{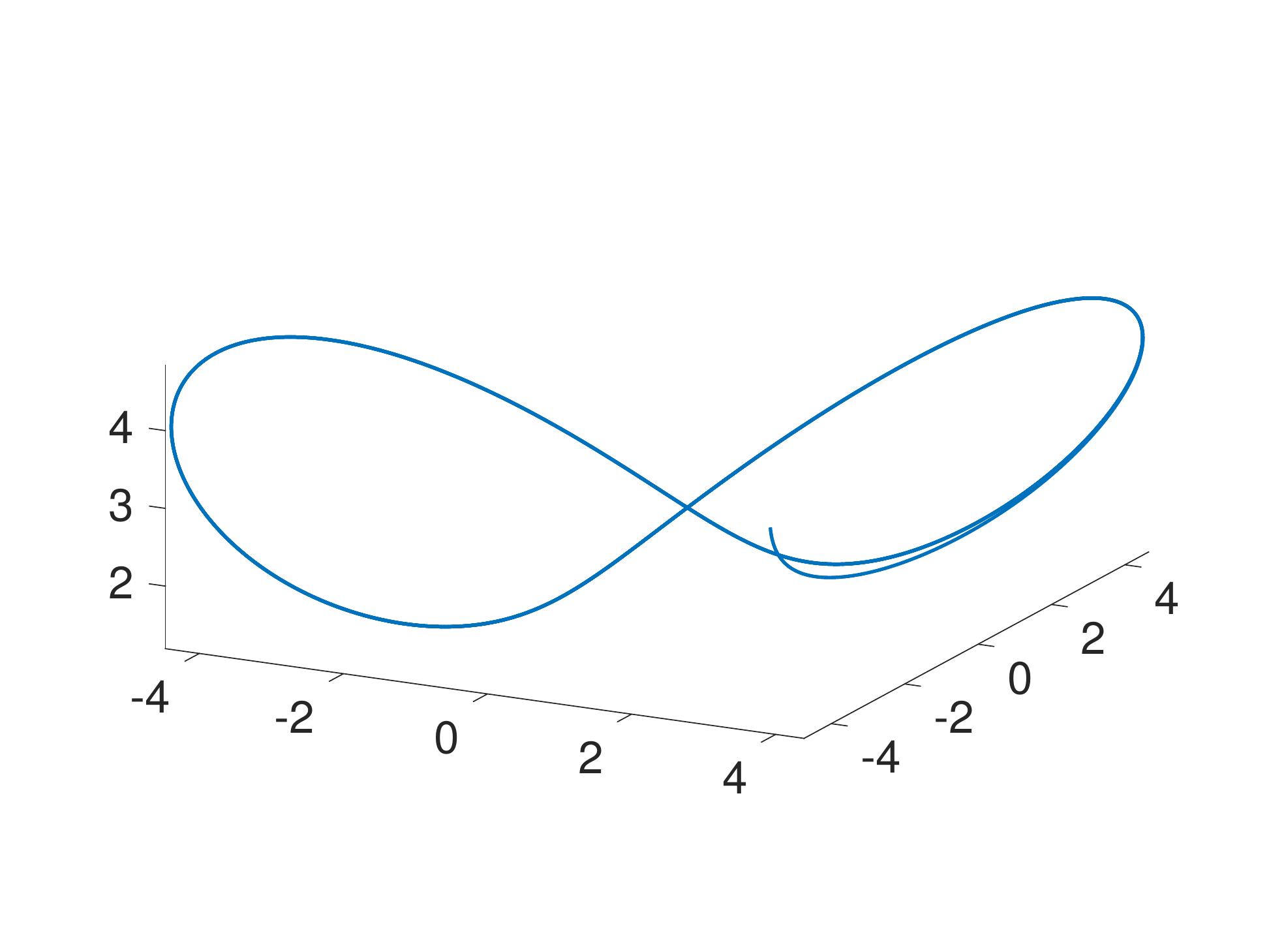}
\includegraphics[width  = 2.1 in]{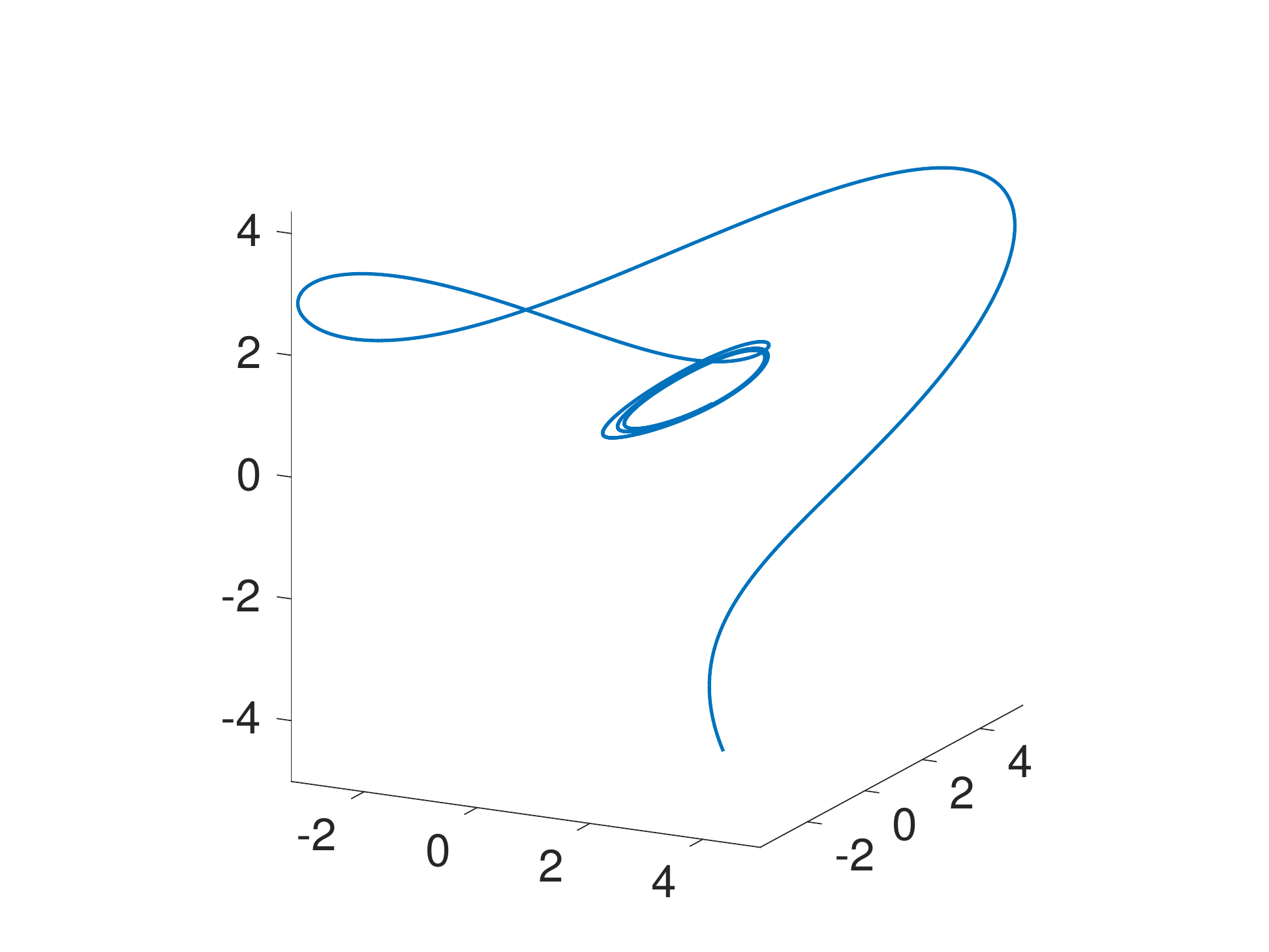}

\caption{Lorenz system: State space plots for different $\alpha$ with $dt = 0.005$. Top: $\alpha = -1$ (left), $\alpha=4.7$ (middle), $\alpha= 6.9$ (right), where the left figure is the usual chaos, the middle is chaotic with periodic windows, and the right one is chaotic. Bottom: $\alpha = 7.075$ (left), $\alpha = 7.73$ (right), where the dynamics include a pitchfork bifurcation and limit cycles, respectively.}
\label{fig:Lorenz3D}
\end{figure}

\noindent \textbf {Lorenz 3D.} We consider the Lorenz system with a single bifurcation parameter $\alpha$:
\begin{equation}
\begin{cases}
\dot{x}_1 &= 10(x_2-x_1)\\
\dot{x}_2 &= -x_1x_3 + (24-4\alpha)x_1 +x_1x_2\\
\dot{x}_3 & = x_1x_3 - \dfrac{8}{3}x_3,
 \end{cases}
 \label{eqn:lorenz}
 \end{equation}
 see \cite{sun2009dynamics}. To validate our approach, we use five data sets  from different bifurcation regimes associated with $\alpha = -1$, $\alpha = 4.7$, $\alpha = 6.9$, $\alpha = 7.075$, and $\alpha = 7.73$. For $\alpha = -1$, the solution is the usual chaotic system with the initial condition set to $U_0 = [-8,7,27]$. For $\alpha = 4.7$, the initial condition is set to $U_0 =[0,-0.01,9]$ and the system exhibits chaos with periodic windows. For $\alpha = 6.9$ with $U_0= [1,2,1]$, the solution exhibits chaos. For $\alpha = 7.075$ with $U_0=[1,1,2]$, the solutions undergoes a pitchfork bifurcation. And lastly, for $\alpha = 7.73$ with $U_0 = [2,1,-5]$, the system has a limit cycle. The time-step is set to $dt = 0.005$ and the finals times are set to: $T = 7.5$, $12.5$, $50$, $15.0 $, and $10.0$ so that the corresponding state spaces exhibit the dynamics as indicated in \cite{sun2009dynamics} (see Figure~\ref{fig:Lorenz3D}). Noise is added to the velocity, with $\sigma_{noise} = 0.5\% $, see Figure~\ref{fig:Lorenz3D_velocity}. 
 
 \medskip
 
 \begin{figure}[h!]
\centering
\includegraphics[width  = 2.1 in]{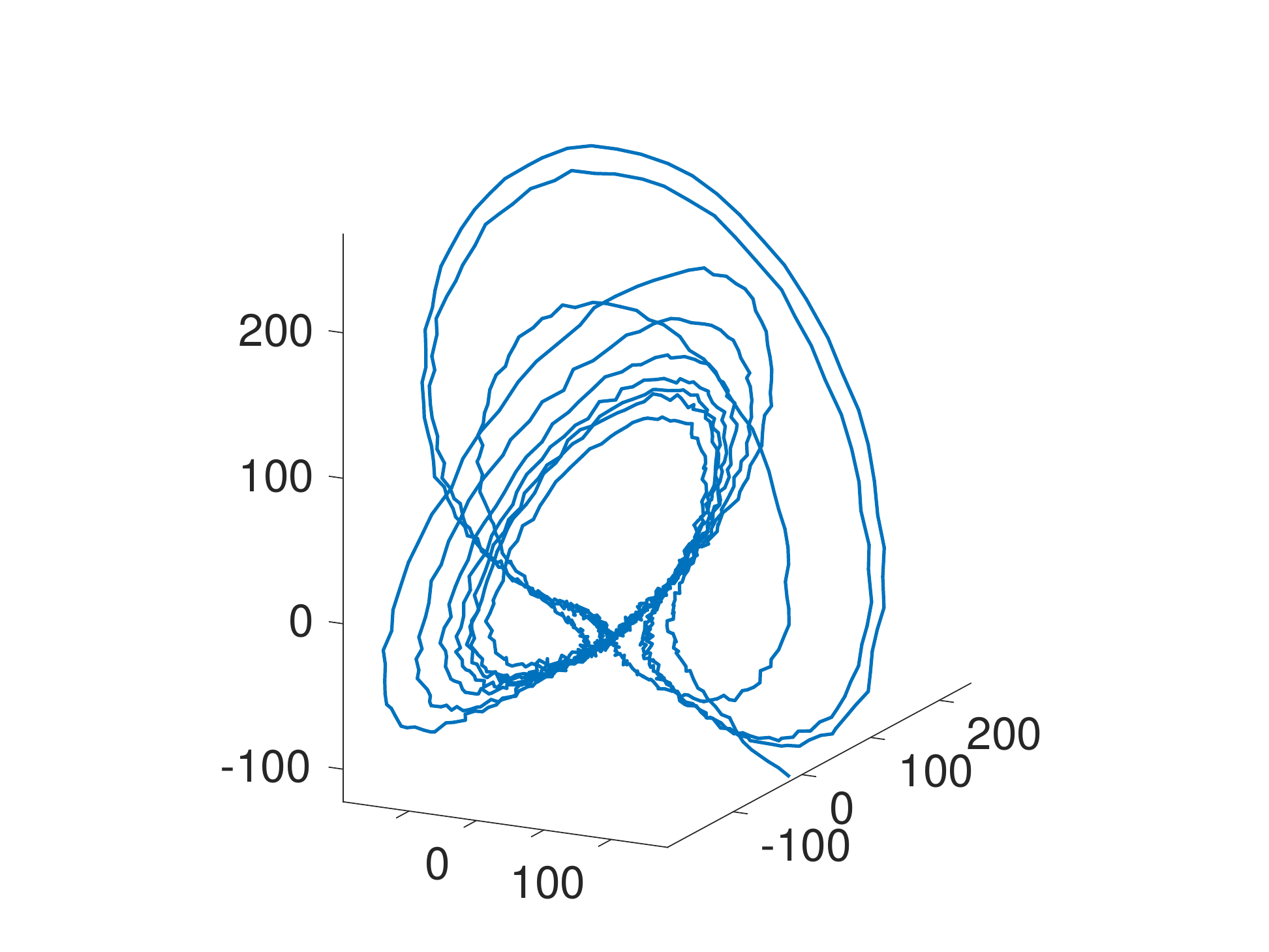}
\includegraphics[width  = 2.1 in]{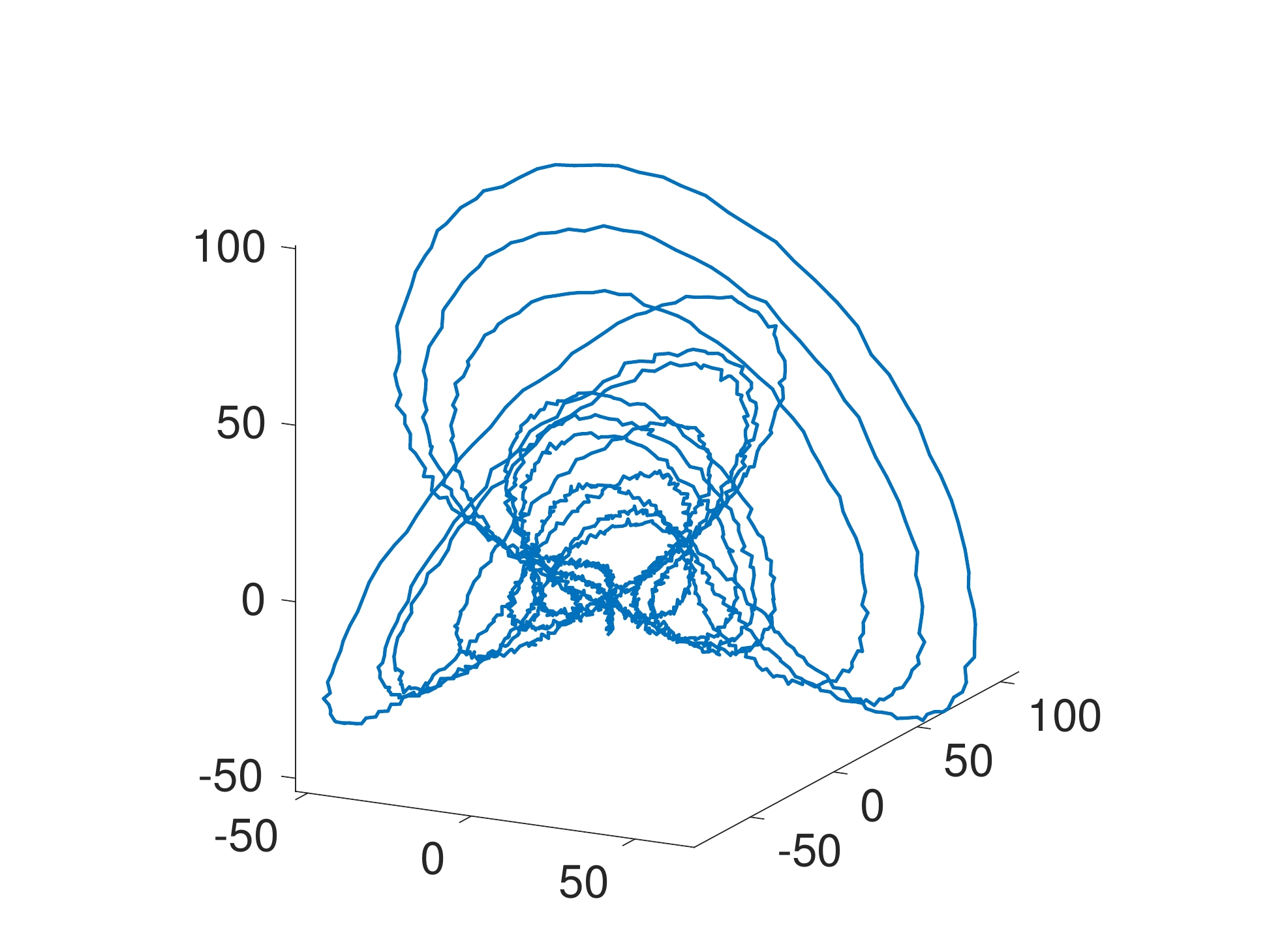}
\includegraphics[width  = 2.1 in]{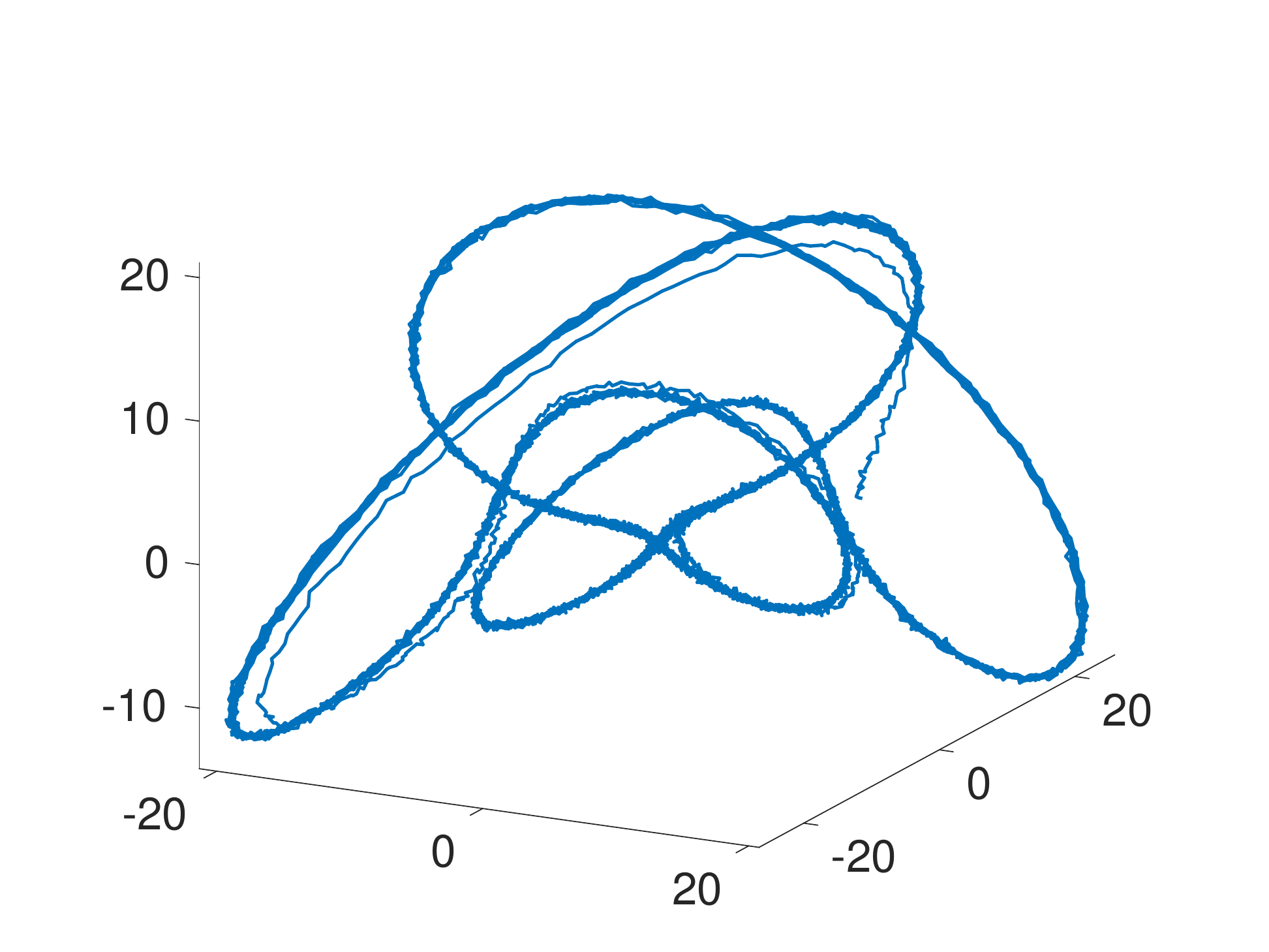}\
\includegraphics[width  = 2.1 in]{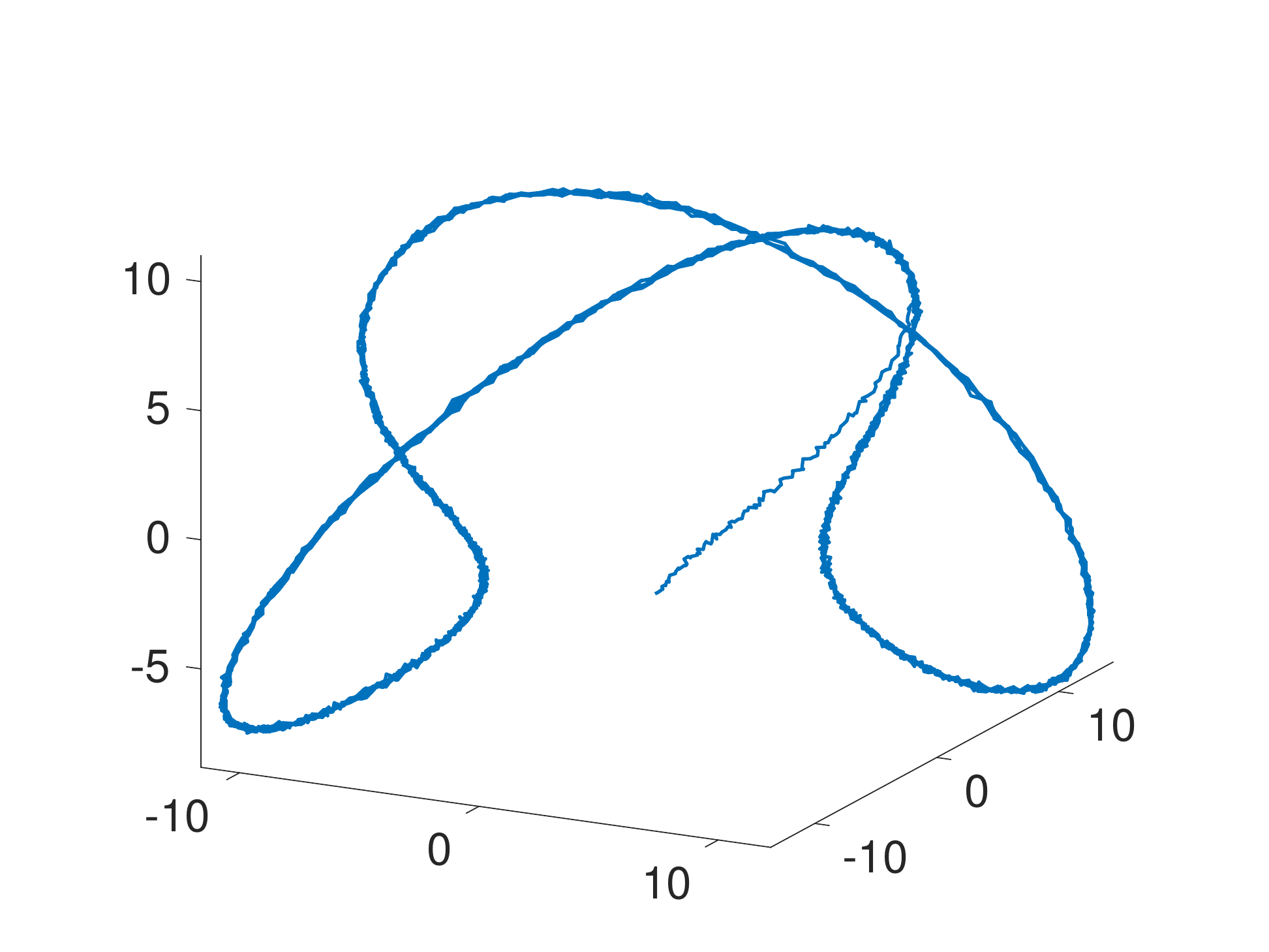}
\includegraphics[width  = 2.1 in]{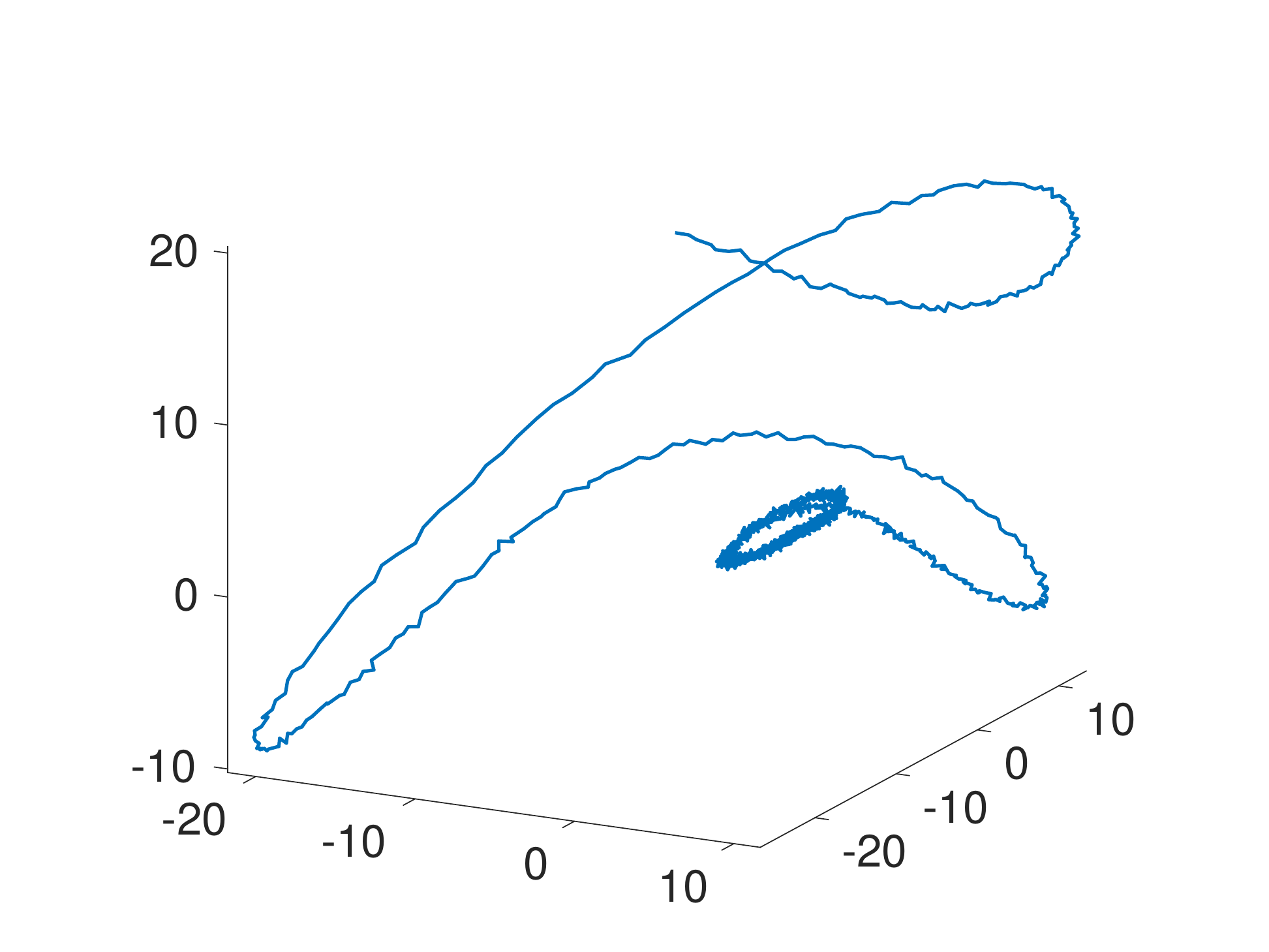}
\caption{Noisy velocity space plots corresponding to the data given in Figure~\ref{fig:Lorenz3D} with noise level $\sigma_{noise} = 0.5\%$. Top: $\alpha = -1$ (left), $\alpha= 4.7$ (middle), $\alpha= 6.9$ (right). Bottom: $\alpha = 7.075$ (left), $\alpha = 7.73$ (right). }
\label{fig:Lorenz3D_velocity}
\end{figure}

\newpage
 
Applying our algorithm to this data with $\delta_{thres} =  1.7$ yields a recovery rate of $P =0.97$. Moreover, the recovered coefficients are close to the true values, with relative error smaller than $3\%$ for $\alpha = -1$ and less than $0.1\%$ for the remaining $\alpha$'s (see Table~\ref{table:Lorenz}). On the other hand, applying the $\ell_0$-penalized model yields less consistent results,  for example in the $\alpha = 7.075$ case the recovery probability is less than $P=0.73$. 
 
\begin{center}
\captionof{table}{Lorenz system. Recovered coefficients from all five sets for the second component 
$\dot{x_2}$. The true values are highlighted in {\color{red}(red)}.} \label{table:Lorenz}
\begin{tabular}{|c|c|c|c|c|c| }
\hline
Coefficients & Set 1 & Set 2 & Set 3 & Set 4 & Set 5  \\
\hline
$1$ & 0 & 0 & 0 &0 & 0  \\
\hline
$x_1$ & 28.0232 ({\color{red} 28})& 5.2104 ({\color{red} 5.2})& -3.6068 ({\color{red} -3.6})& -4.2960 ({\color{red} -4.3})& -6.9246 ({\color{red} -6.92})\\
\hline
$x_2$ &-1.0093 ({\color{red} -1.0})& 4.6970 ({\color{red} 4.7}) & 6.9020 ({\color{red} 6.9})& 7.0719 ({\color{red} 7.075})& 7.7310 ({\color{red} 7.73}) \\
\hline
$x_3$ &0 & 0 & 0 &0 & 0 \\
\hline
$\vdots$ &$\vdots$ &$\vdots$ &$\vdots$ &$\vdots$ &$\vdots$ \\
\hline 
$x_1x_3$ & -1.0002 ({\color{red} -1}) & -1.0003 ({\color{red} -1})& -0.9989 ({\color{red} -1}) & -1.0002 ({\color{red} -1}) & -0.9992 ({\color{red} -1})\\
\hline
$\vdots$ &$\vdots$ &$\vdots$ &$\vdots$ &$\vdots$ &$\vdots$ \\
\hline
$x_3^4$ & 0 & 0 & 0 &0 & 0\\
\hline
\end{tabular}
\end{center}

\bigskip

\begin{figure}[h!]
\centering
\includegraphics[width = 4in]{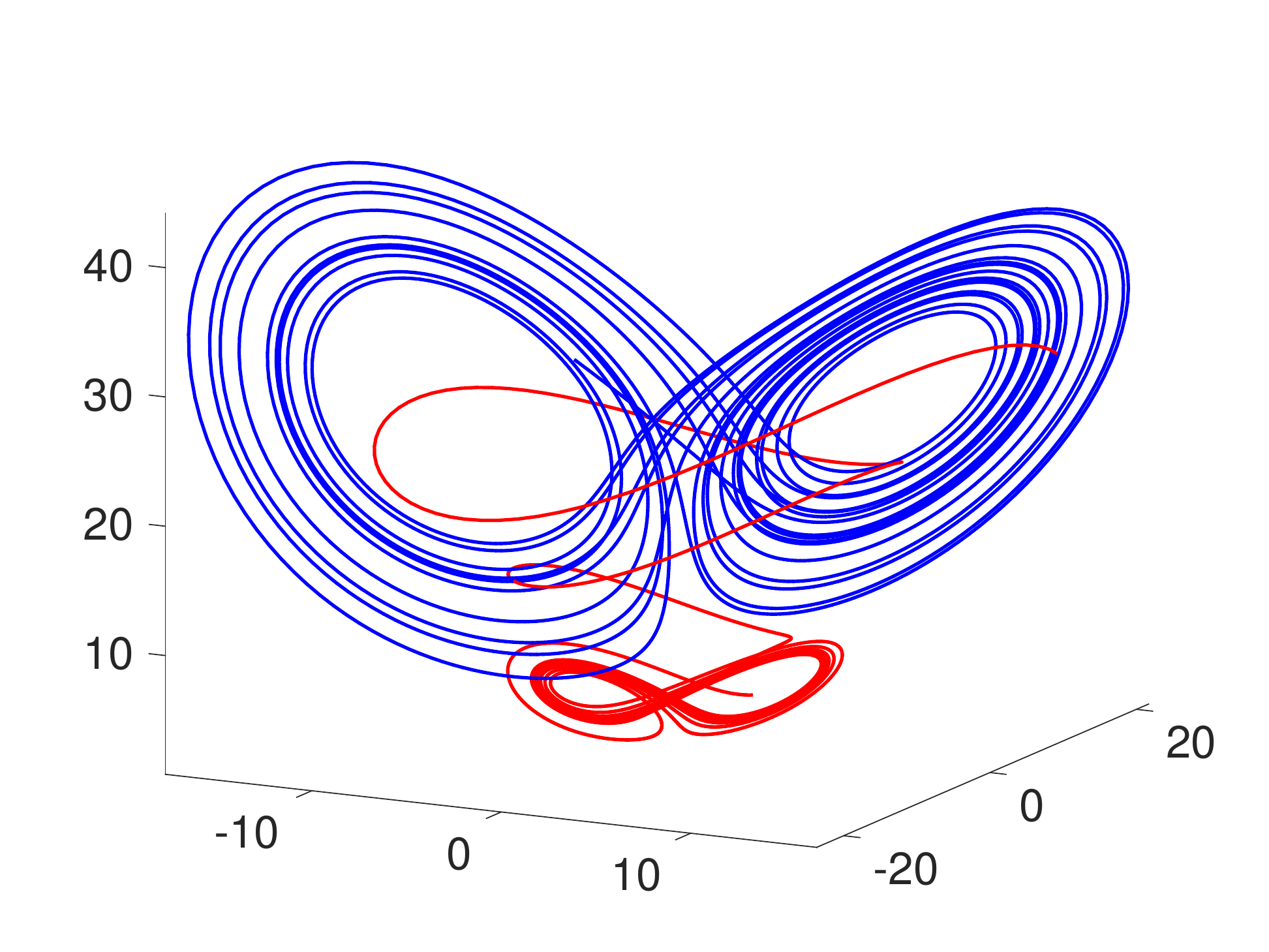}
\caption{Switching system: State space plot of the Lorenz system (Equation~\eqref{eqn:lorenz}), where the bifurcation parameter $\alpha$ switches from $-1$ (blue curve) to $6.6$ (red curve). }
\label{fig:switching}
\end{figure}

\noindent\textbf{Switching Systems.} For our last example, we illustrate a potential application of the proposed method to switching systems.
Consider the Lorenz system, Equation~\eqref{eqn:lorenz}, where the parameter $\alpha$ changes from $-1$ to $6.6$ at some unknown time. The state space of the whole system is plotted in Figure ~\ref{fig:switching}. Since the location of the parameter change is unknown (as well as the underlying model), we break the entire trajectory into $M$ sub-trajectories and consider each of these sub-trajectories as our different sources, $X^{(i)}$. Therefore, this problem fits within our framework-- the learned parameters are allowed to vary between each sub-trajectory. In this example we take $M=32$. The recovered coefficients for the second component $\dot{y}$ are plotted in Figure~\ref{fig:switching_coef}, where the terms $x$, $y$, and $xz$ are correctly identified in all sub-trajectories except the one that contains the switch.

\begin{figure}[h!]
\centering
\includegraphics[width = 4 in]{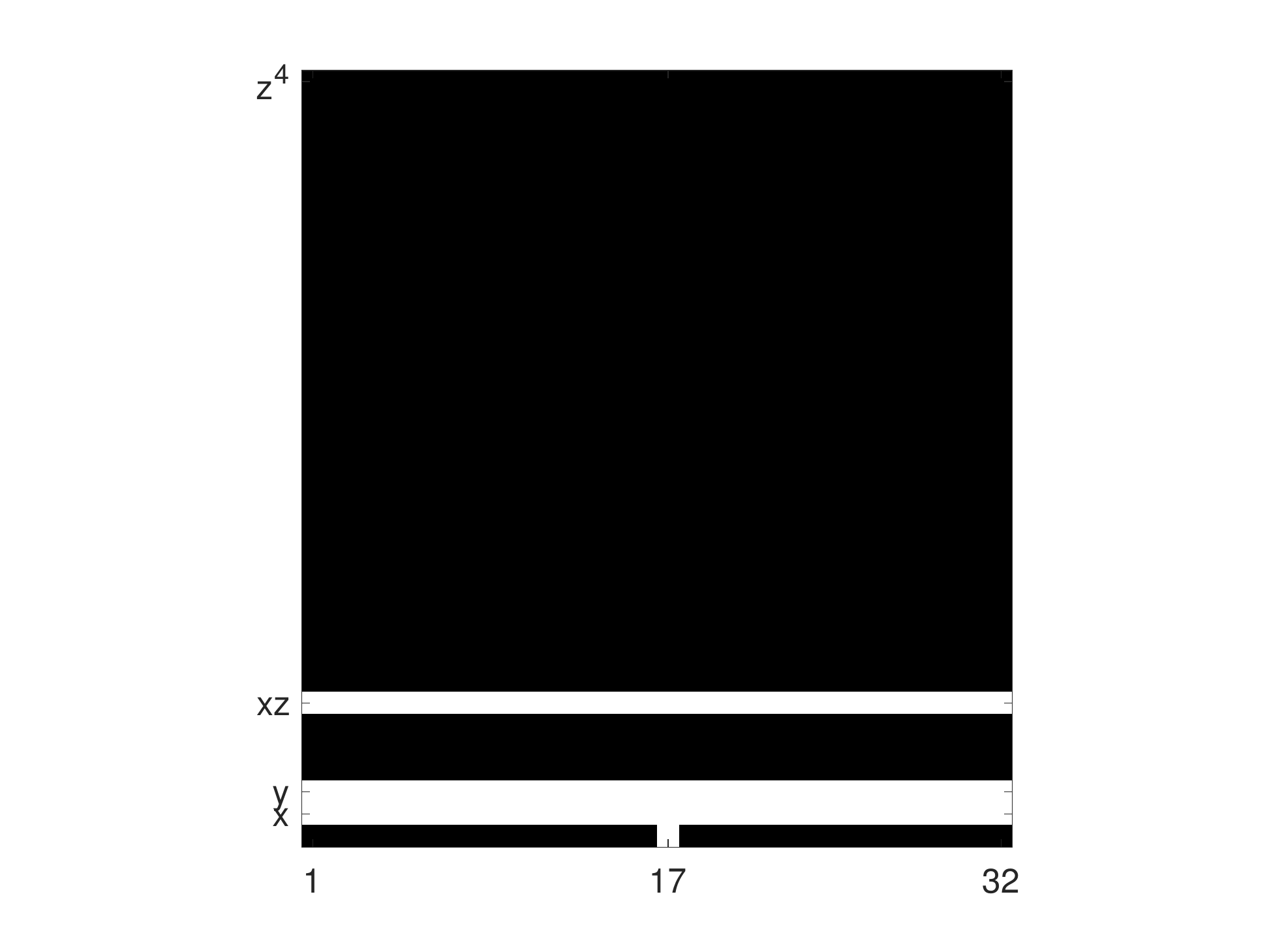}
\caption{Switching System: Indices of the recovered coefficients for the second component of the switching system from Figure \ref{fig:switching}. Our method correctly identifies the terms in the governing equation as well as the location of the switch (the anomalous sub-trajectory at 17).}
\label{fig:switching_coef}
\end{figure}

\section{Conclusion}\label{section:conclusion}
 We presented a method for extracting governing equations from multiple sources of data using a group-sparsity constraint as well as developed a new group-thresholding algorithm to solve our proposed optimization problem. Our main contribution is the use of group sparsity  for learning physical laws from multiple data sources which are controlled by the same mathematical model with different parameters. We also provide convergence guarantees for the associated regression problem. Lastly, convergence of our algorithm to a local minimizer is detailed.

\section*{Acknowledgments}
H.S. acknowledges the support of AFOSR, FA9550-17-1-0125. R.W. and G.T. acknowledge the support of NSF CAREER grant $\#1255631$. The authors would like to thank the CNA and NSF for their support of the  CNA-KiNet Workshop: ``Dynamics and Geometry from High Dimensional Data" at Carnegie Mellon University in March 2017, where this work was finalized and presented.

\section{Appendix}

Proof of Theorem~\ref{thrm:convergence}.\\

Part 1:
 Let $\widetilde{C}^{k+1}$ be defined as:
 $$\widetilde{C}^{k+1}:=H_{\sqrt{\gamma}}\left(C^{k} + D^T\star(V -  D\star C^{k} )\right),$$ then
we can show that $F(C^{k+1}) \leq F(C^{k})$ by adapting some of the arguments from \cite{blumensath2008iterative}. The objective function at $C^{k+1}$ can be bounded by:
\begin{align*}
F(C^{k+1})&=\| D\star C^{k+1} - V \|_2^2 +  \gamma \|C^{k+1}\|_{2,0}\\
&=\| D\star C^{k+1} - V \|_2^2 +   +  \gamma \|\widetilde{C}^{n+1}\|_{2,0}\\
&\leq\| D\star \widetilde{C}^{k+1} - V \|_2^2  +  \gamma \|\widetilde{C}^{k+1}\|_{2,0}=F(\widetilde{C}^{k+1})
\end{align*}
where the second line comes from the fact that $C^{k+1}$ and $\widetilde{C}^{k+1}$ have the same support set (and thus the same value with respect to $\ell^{2,0}$) and the third line comes from the restricted least-squares update (second step of Equation~\eqref{eq:updatematrix}). Let $A^{(i)}:=\textbf{I}-  \left(D^{(i)}\right)^T D^{(i)}$ and assume that the eigenvalues of $A^{(i)}$, denoted as $\lambda^{(i)}$, are bounded away from zero and are less than 1, i.e.  $\lambda^{(i)}\in[\overline{\lambda},1]$ for $\overline{\lambda}>0$. Define the norm with respect to $A$ as $\| - \|_{2,A}$, then
\begin{align*}
F(\widetilde{C}^{k+1})&\leq F(\widetilde{C}^{k+1})+ \|\widetilde{C}^{k+1}-C^{k} \|^2_{2,A^{(i)}}\\
&=\| D\star\widetilde{C}^{k+1} -V \|_2^2 +\gamma\, \|\widetilde{C}^{k+1}\|_{2,0} + \|\widetilde{C}^{k+1}-C^{k} \|^2_2 -\|D\star(\widetilde{C}^{k+1}-C^{k}) \|^2_{2} \\
&=F^*(\widetilde{C}^{k+1},C^{k})\\
&=\argmin{C} \ F^*(C,C^{k})\\
&\leq F^*(C^{k},C^{k})\\
&= F(C^{k}).
\end{align*}
Note that this argument also shows that $F(\widetilde{C}^{k+1}) \leq F(\widetilde{C}^{k})$. This implies that the energy $F$ converges. 

\medskip

Part 2: We show convergence to a local minimizer when the dictionary is coercive. To do so, consider the finite sum:
$$\sum\limits_{k=0}^N \| {C}^{k+1}-C^{k} \|_2^2 $$
which increases monotonically with respect to $N$. The sum is bounded by:
\begin{align*}
\sum\limits_{k=0}^N \| C^{k+1}-C^{k} \|_2^2 \leq \sum\limits_{k=0}^N \| C^{k+1}-\widetilde{C}^{k+1} \|_2^2+ \| \widetilde{C}^{k+1}-C^{k} \|_2^2 
\end{align*}
which we will individually bound as follows. The first term is bounded by:
\begin{align*}
\sum\limits_{n=0}^N \| \widetilde{C}^{k+1}-C^{k} \|_2^2 &= \sum\limits_{k=0}^N \sum\limits_{i=1}^m \ \left\|\left(\widetilde{c}^{(i)}\right)^{k+1}-\left({c}^{(i)}\right)^{k} \right\|^2_{2}\\
&\leq \bar{\lambda}^{-1}\, \sum\limits_{k=0}^N \sum\limits_{i=1}^m \ \left\|\left(\widetilde{c}^{(i)}\right)^{k+1}-\left({c}^{(i)}\right)^{k} \right\|^2_{2,A^{(i)}}\\
&\leq \bar{\lambda}^{-1}\, \sum\limits_{k=0}^N \left( F(C^{k})-F(C^{k+1})\right)\\
&= \bar{\lambda}^{-1}\, \left( F(C^{0})-F(C^{N+1})\right)\\
&\leq \bar{\lambda}^{-1}\, F(C^{0}).
\end{align*}
To bound the second term, consider the norm restricted onto the support set $S^{k+1}$:
\begin{align*}
\sum\limits_{k=0}^N \| C^{k+1}-\widetilde{C}^{k+1} \|_2^2 = \sum\limits_{k=0}^N \sum\limits_{i=1}^m \ \left\|\left({c}^{(i)}\right)^{k+1}-\left(\widetilde{c}^{(i)}\right)^{k+1} \right\|^2_{2}= \sum\limits_{k=0}^N \sum\limits_{i=1}^m \ \left\|\left({c}^{(i)}\right)^{k+1}-\left(\widetilde{c}^{(i)}\right)^{k+1} \right\|^2_{2 \, | \,S^{k+1}}.\\
\end{align*}
If the matrix $D$ is coercive over  $S^{k+1}$, with coercivity constant $\delta>0$, then
\begin{align*}
 \sum\limits_{k=0}^N \sum\limits_{i=1}^m \ \left\|\left({c}^{(i)}\right)^{k+1}-\left(\widetilde{c}^{(i)}\right)^{k+1} \right\|^2_{2 \, | \,S^{k+1}}     
 & \leq \delta^{-1}\, \sum\limits_{k=0}^N \sum\limits_{i=1}^m \ \left\| D^{(i)}|_{S^{k+1}} \left( \left({c}^{(i)}\right)^{k+1}-\left(\widetilde{c}^{(i)}\right)^{k+1} \right) \right\|^2_{2 \, | \,S^{k+1}}\\
  & \leq \delta^{-1}\, \sum\limits_{k=0}^N \left( F(\widetilde{C}^{k+1} ) -  F({C}^{k+1} )\right)\\
    & \leq \delta^{-1}\, \sum\limits_{k=0}^N \left( F(C^{k} ) -  F({C}^{k+1} )\right)\\
    & \leq \delta^{-1}\,F(C^{0} )
    \end{align*}
Combining these two bounds yields:
\begin{align*}
\sum\limits_{k=0}^N \| C^{k+1}-C^{k} \|_2^2 \leq \sum\limits_{k=0}^N \| C^{k+1}-\widetilde{C}^{k+1} \|_2^2+ \| \widetilde{C}^{k+1}-C^{k} \|_2^2 \leq (\bar{\lambda}^{-1}+\delta^{-1})\, F(C^{0}).
\end{align*}
 Therefore, for any $\epsilon>0$, there exist an $N>0$ such that for all $k>N$, we have $\| C^{k+1}-C^{k} \|_2 \leq \epsilon$. Using this condition and an analogy of Lemma 3.4 from \cite{blumensath2008iterative} (changing element-wise to row-wise) yields a subsequence $C^k$ converging to a local minimizer.

%

\end{document}